\documentclass[11pt,a4paper]{article}
\usepackage{latexsym}
\usepackage{amssymb}
\usepackage{amsthm}
\usepackage{xcolor}
\usepackage{multicol}
\usepackage{amsmath}
\title{On the Alesker-Verbitsky conjecture on hyperK\"ahler manifolds}
\author{Sławomir Dinew and Marcin Sroka}
\date{\small \textit{Dedicated to Professor Sławomir Kołodziej on the occasion of his 60th birthday}} 

\usepackage{geometry}
\newtheorem{theorem}{Theorem}[section]
\newtheorem{proposition}[theorem]{Proposition}
\newtheorem{lemma}[theorem]{Lemma}
\newtheorem{remark}[theorem]{Remark}

\newtheorem{definition}[theorem]{Definition}
\newtheorem{conjecture}[theorem]{Conjecture}
\newcommand{\hh}{\mathbb{H}}
\newcommand{\rr}{\mathbb{R}}
\newcommand{\cc}{\mathbb{C}}
\newcommand{\hn}{\mathbb{H}^n}

\newcommand{\ii}{\mathfrak{i}}
\newcommand{\jj}{\mathfrak{j}}
\newcommand{\kk}{\mathfrak{k}}

\newcommand{\pa}{\partial}
\newcommand{\bpar}{\overline{\partial}}
\newcommand{\jpar}{\partial_J}
\newcommand{\bjpar}{\overline{\partial_J}}
\newcommand{\ob}{\nabla^{Ob}}

\newcommand{\bOmega}{\overline{\Omega}}
\newcommand{\Pf}{{\text{Pf }}}
\newcommand{\tr}{\text{tr}}
\begin{document}
\maketitle

\textbf{Abstract:} We solve the quaternionic Monge-Amp\`ere equation on hyperK\"ahler manifolds. In this way we prove the ansatz for the conjecture raised by Alesker and Verbitsky claiming that this equation should be solvable on any hyperK\"ahler with torsion manifold, at least when the canonical bundle is trivial holomorphically. The novelty in our approach is that we do not assume any flatness of the underlying hypercomplex structure which was the case in all the approaches for the higher order a priori estimates so far. The resulting Calabi-Yau type theorem for HKT metrics is discussed.    



\begin{center} \tableofcontents \end{center}

\section{Introduction}

The quaternionic Monge-Amp\`ere equation on compact HKT manifolds was introduced by Alesker and Verbitsky in \cite{AV10}. On a general hyperhermitian manifold $(M,I,J,K,g)$ of quaternionic dimension $n$ it takes the form
\[ \label{qma} \tag{1.1} \begin{gathered} (\Omega + \pa \jpar \phi )^n = e^f \Omega^n, \\ \Omega + \pa \jpar \phi > 0. \end{gathered}\]

The goal of this paper is to prove the higher order estimates for the quaternionic Monge-Amp\`ere equation (\ref{qma}) on compact hyperK\"ahler manifolds. The highlight of our result is that we do not assume any flatness or additional integrability of the underlying hypercomplex structure as was always the case for the higher order estimates known so far. The hyperK\"ahler condition, as will be seen during the derivation of the estimates, may in turn be interpreted as a curvature condition which we hope can be removed. As a result we solve this equation on any compact hyperK\"ahler manifold, cf. Remark \ref{comment}.

Our main result is stated as follows:  

\begin{theorem}\label{mainthm}
Let $(M,I,J,K,g)$ be a compact, connected hyperK\"ahler manifold. For any real function $f \in C^\infty(M)$ there exists a unique, up to the addition of a constant, smooth solution $\phi$ to the quaternionic Monge-Amp\`ere equation (\ref{qma}) provided the necessary normalization condition
\[\label{norm} \tag{1.2} \int_M e^f \Omega^n \wedge \bOmega^n = \int_M \Omega^n \wedge \bOmega^n\] is satisfied.
\end{theorem}
Equation (\ref{qma}) was motivated by the prior research in the local case, cf. \cite{A03}, and the attempt to prove the analog of the Calabi conjecture in quaternionic geometry, cf. \cite{AV10, V09}, as we explain in the next section. 
It was conjectured in \cite{AV10} that equation (\ref{qma}) can always be solved at least in the case when the canonical bundle of the HKT manifold is trivial holomorphically. 
Given the recent progress in solving Calabi-Yau type equations for non K\"ahler metrics, cf. \cite{GL10, TW10b, SzTW17, TW17, Sz18, TW19}, this is expected to hold even without the latter assumption, cf. \cite{AS17, Sr19}. To sum up Theorem \ref{mainthm} constitutes the ansatz for proving the following conjecture.

\begin{conjecture}{\cite{AV10}} \label{qCc}
Let $(M,I,J,K,g)$ be a compact, connected HKT manifold. Suppose that there exists a non vanishing $I$-holomorphic $(2n,0)$ form  $\Theta$ on $M$. For every real smooth function $f$ the quaternionic Monge-Amp\`ere equation (\ref{qma}) admits a unique, up to the constant, smooth solution provided the function $f$ satisfies the necessary condition \[ \label{normgen} \tag{1.3} \int_M (e^f-1) \Omega^n \wedge \overline{\Theta}=0.\]
\end{conjecture} 

Let us just mention that the name for the equation (\ref{qma}) is justified as follows. Consider the quaternionic variables \[ \tag{1.4} \label{rc} q_i=x_{4i}+x_{4i+1} \ii + x_{4i+2} \jj + x_{4i+3} \kk, \] in the flat quaternionic space $\hn$ of quaternionic dimension $n$. There are the so called Cauchy-Riemann-Fueter derivatives 
\[ \label{crf} \tag{1.5} \frac{\partial \phi}{ \partial \bar{q}_{\alpha}}
= \frac{\partial \phi}{\partial x_{4\alpha}} + \ii \frac{\partial \phi}{\partial x_{4\alpha+1}} + \jj \frac{\partial \phi}{\partial x_{4\alpha+2}} + \kk \frac{\partial \phi}{\partial x_{4\alpha+3}}, \] 
\[ \label{crfb} \tag{1.6} \frac{\partial \phi}{ \partial q_{\alpha}}= \frac{\partial \phi}{\partial x_{4 \alpha}} - \frac{\partial \phi}{\partial x_{4 \alpha +1}} \ii -  \frac{\partial \phi}{\partial x_{4 \alpha +2}} \jj - \frac{\partial \phi}{\partial x_{4 \alpha +3}} \kk . \] 
It may be checked, cf. \cite{AV06} or \cite{Sr18} for an elementary calculation, that in this case
\[ \label{fqma} \tag{1.7} \big(\pa \jpar \phi \big)^n = \det \Bigg[ \frac{\partial^2 \phi}{ \partial q_\alpha \partial \bar{q_{\beta}}} \Bigg]_{\alpha, \beta} \Theta, \] where $\Theta$ is the canonical trivialization of $K_I(\hn)$ and the determinant has
to be understood in a proper way - as the Moore determinant, cf. \cite{M22}, of the hyperhermitian matrix. The Dirichlet problem for the operator (\ref{fqma}) was first considered by Alesker in \cite{A03} where continuous solutions were found for continuous
right hand sides. After that, the problem was solved in the smooth category by Zhu \cite{Z17}. Later, thanks to the form of (\ref{fqma}), the pluripotential approach was taken up resulting in providing continuous solutions even for right hand sides in $L^p$ spaces.
For $p \geq 4$ it is due to Wan, cf. \cite{W20}, and for $p>2$ due to the second named author, cf. \cite{Sr18}. In the latter case the exponent was proven to be optimal. This equation is also covered by the very general approach taken up in the last
two decades by Harvey and Lawson. They provide viscosity solutions for the Dirichlet problem even for domains in quaternionic manifolds, cf. \cite{HL09c, HL11, HL18}.    

Coming back to the advances towards proving Conjecture \ref{qCc}. Generally speaking the difficulties with obtaining a priori estimates for the equation (\ref{qma}) are caused, as we explain in depth in the next section, by the fact that a generic
hypercomplex structure locally is not the pull back of the flat structure from $\hn$. The lack of {\it quaternionic} coordinates forces one to work in the general case, at best, with  holomorphic coordinates for the reference complex structure $I$.
But this in turn results in equation (\ref{qma}) depending not only on the coefficients of the metric tensor but also of the endomorphism field $J$ since, in the holomorphic coordinates for $I$, 
\[ \tag{1.8} \label{of} \Omega + \pa \jpar \phi := \Omega^\phi_{ij} dz_i \otimes dz_j= \Big( -(g_{i \bar{k}} + \phi_{i \bar{k}})J^{\bar{k}}_j + ( g_{j \bar{k}} + \phi_{j \bar{k}})J^{\bar{k}}_i \Big) dz_i \otimes dz_j.\] 

Proving estimates in such coordinates is similar to solving the complex Monge-Amp\`ere equation by performing the calculations in generic real coordinates. Another drawback is that equation (\ref{qma}), instead of being of the form (\ref{fqma}), is an equation on the Pfaffian of the coefficients of the two form $\Omega + \pa \jpar \phi$ in the complex coordinates
\[ \tag{1.9} \label{cpe} \Pf \big[\Omega^\phi_{ij}\big]_{i,j} = e^f \cdot \Pf \big[\Omega_{ij}\big]_{i,j}. \]  
As will be seen in the calculations, the more times the equation (\ref{cpe}) is differentiated the more the lack of quaternionic derivatives becomes an issue. 

In the paper \cite{AV10}, the $C^0$ estimates for the solutions to the equation (\ref{qma}) were derived by the classical Moser iteration method in the setting exactly as in Conjecture \ref{qCc}. There the existence of the holomorphic trivialization is
crucial for the argument to work. Later it was shown that the $C^0$ estimate holds for (\ref{qma}) even without the assumption on the holomorphic triviality of $K_I(M)$, cf. \cite{AS17, Sr19}, but the methods are much more involved.
As for the higher order estimates, as we mentioned, they have been derived only under the assumption that the hypercomplex structure is integrable in the strong sense, i.e. is locally flat. In addition, either the initial metric should be very special or
the manifold has to admit in addition a hyperK\"ahler metric compatible with the flat hypercomplex structure. Precisely, in \cite{A13} it was shown that in the case the manifold is a torus or its quotient endowed with the flat hyperK\"ahler metric (this implies 
in particular the flatness of the hypercomplex structure), the Laplacian bound on the solution of (\ref{qma}) holds. Alesker proved also that the analogue of the Evans--Krylov theorem, cf. \cite{E82}, holds under the assumption of flatness of
the hypercomplex structure. In \cite{GV21} the authors show the Laplacian bound, unlike in Alesker's result depending on the gradient bound, for equation (\ref{qma}) on certain eight dimensional nilmanifolds endowed with the flat 
hypercomplex structures and torus action invariant initial HKT metrics. Shortly before this preprint was written down the preprint \cite{BGV21} appeared on arXiv. The authors take up there the parabolic approach for the equation (\ref{qma}) but the assumption under which they are able to prove the convergence of the flow to the solution are exactly as in Alesker's paper \cite{A13}.              

In the current note we carry out the computations for all the higher order estimates in geodesic coordinates for the Obata connection. This seems to be the main technical input which allows us to overcome the issues coming
from the dependence of the equation (\ref{qma}) on the second complex structure $J$. This is still not enough though to deal with the dependence on the metric. Because of that the hyperK\"ahler assumption appears. 

Our strategy in this paper is as follows. First of all, motivated by an influential idea of Błocki from \cite{B09} (see also \cite{Gu}), we show, cf. Theorem \ref{c1b}, the gradient estimate in the setting of Theorem \ref{mainthm}.
Here, the hyperK\"ahler condition plays similar role as non negativity of the holomorphic bisectional curvature of the background K\"ahler metric in the case of the complex Monge-Amp\`ere equation. This is somewhat  surprising since the hyperK\"ahler metric
is not necessarily of such a curvature. In the general hypercomplex case the gradient estimate does not seem to follow from arguments as in \cite{B09}. The troubles in this case are caused partially by the failure of the Leibniz rule for the operators (\ref{crf})
and (\ref{crfb}). Let us mention that (direct) gradient estimate is not known for general complex Hessian equations. We use Theorem \ref{c1b} in Section 5 to derive, with the aid of Theorem \ref{le}, the full $C^2$ estimate. In this case the proof is standard
and relies on an idea of Błocki from \cite{Bnotes}. The crucial result is Theorem \ref{le} proven in Section 4 which gives the bound on the Laplacian, or equivalently the quaternionic Hessian $\pa \jpar \phi$, for the solutions of (\ref{qma}). 
Let us remark that for general HKT metric a major issue is how to handle third order terms. Roughly speaking the positive term appearing, being formed by the squares of sums of certain third order derivatives, compensates only half of the negative term. At the moment we do not know how to deal with this difficulty
and the classical methods of \cite{Y78, B09, GL10} does not seem, cf. Remark \ref{lapdisc}, to work. The hyperK\"ahler assumption allows us to get rid of some terms coming from differentiating the metric coefficients and in turn, by considering a more general perturbation than the classical one in the Pogorelov approach for the Laplacian bound, cf. \cite{Y78, B09, GL10}, we are able to ignore the negative term.      
  
As we explain in details in the next section, Theorem \ref{mainthm} allows one to draw some conclusions concerning the HKT geometry of the underlying hyperK\"ahler manifold. First of all having a hypercomplex manifold $(M,I,J,K)$ any hyperhermitian metric $g$ provides a canonically associated smooth section of the canonical bundle $K_I(M)$
 via the map
\[ \label{cplexvol} \tag{1.10} g \longmapsto \Omega^n.\] 
Sections obtained in (\ref{cplexvol}) satisfy the properties, defined rigorously in the next section, which we call positivity and $J$-realness.

\begin{proposition} \label{props}
Let $(M,I,J,K,g)$ be a compact, connected hyperK\"ahler manifold. Given any positive and $J$-real trivialization $\Theta$, i.e. of the form $e^f\Omega^n$ for some smooth function $f$, of the canonical bundle $K_I(M)$ there is an HKT metric $\hat{g}$ such that
the associated HKT form $\hat{\Omega}$ satisfies \[ \label{pcv} \tag{1.11} \hat{\Omega}^n = \Theta. \] What is more the metric $\hat{g}$ may be chosen so that the associated HKT form is of the form 
$$a \big(\Omega + \pa \jpar \phi \big)$$
for a smooth function $\phi$ and a positive constant $a$. \end{proposition}

From Proposition \ref{props} it is easy to obtain the so called Calabi-Yau type theorem for HKT metrics, yet only on hyperK\"ahler manifolds.

\begin{remark}\label{vol}
Let $(M,I,J,K,g)$ be a compact, connected hyperK\"ahler manifold. Given any positively oriented volume form $\sigma$ on $(M,I)$ there is an HKT metric $g_\phi$, such that the associated HKT form is $a \big( \Omega + \pa \jpar \phi \big)$, for a smooth function $\phi$ and a positive constant $a$, such that \[ \label{pv} \tag{1.12} \omega_{I,g_\phi}^{2n} = \sigma. \] \end{remark}

A classical calculation allows us to obtain the following as a result of Remark \ref{vol}.

\begin{proposition} \label{pbc}
Let $(M,I,J,K,g)$ be a compact, connected hyperK\"ahler manifold. Given any representative \[ \rho \in c_1^{BC}(M,I) \] of the first Bott-Chern class of $(M,I)$, see \cite{T15}, there is an HKT metric $g_\phi$, whose associated HKT form is $\Omega+ \pa \jpar \phi$, such that
\[ \label{prf} \tag{1.13} Ricc(\nabla^{Ch}_{I,g_\phi}) = \rho.\] In (\ref{prf}) the symbol $Ricc(\nabla^{Ch}_{I,g_\phi})$ denotes the well known Chern--Ricci form associated to the Chern connection of the hermitian structure $(I, g_\phi)$ on $M$.  
\end{proposition}

As the non constant conformal deformation of an HKT metric is never an HKT metric the last three results stated above are non trivial, even under the hyperK\"ahler assumption. Provided one can remove the extra assumption on the initial metric
the Calabi-Yau theorem for yet another class of metrics would be settled. Proving the Calabi-Yau type theorem for different class of hermitian (non K\"ahler) metrics was a subject of an intense study in the last decade. In the classical case of K\"ahler metrics 
it is known due to Yau \cite{Y78}. For the class of Gauduchon metrics it was settled only recently in \cite{SzTW17}, building on \cite{Sz18, TW19}, confirming in turn an old conjecture of Gauduchon from '80s. Actually the same result was proven
there for strongly Gauduchon metrics as well. In \cite{TW17} a Calabi-Yau type theorem, Corollary 1.3 in there, for balanced metrics was proven, yet like in our case, under an extra assumption that the manifold admits a K\"ahler metric. In this case the result follows of course from the Calabi-Yau theorem for K\"ahler metrics as well. The assumption of admitting K\"ahler metric was later relaxed to admitting merely an Astheno-K\"ahler one, cf. \cite{SzTW17}, in which case the classical Calabi-Yau theorem can not be applied. It is still an open problem whether the Calabi-Yau type theorem for the class of balanced metric holds in general, cf. \cite{TW19, SzTW17}. 

Let us finish this introduction by the remark in the spirit of the mentioned Harvey and Lawson theory. Like in the local case, in the global situation the equation (\ref{qma}) is a companion of the real and complex Monge-Amp\`ere equations which received
great attention in the last century and, in certain forms, are related to the fundamental differential geometric problems. 

The complex Monge-Amp\`ere equation on complex $n$ dimensional hermitian manifold $(M,I,g)$ taking the form
\[ \tag{1.14} \label{cma} (\omega + \ii \pa \bpar \phi)^n = e^f \omega^n\] was solved, as we mentioned, on K\"ahler manifolds by Yau \cite{Y78} and on general hermitian manifold by Tosatti and Weinkove \cite{TW10a, TW10b}, cf. also \cite{GL10}.
It was proven that it can be solved even on almost complex manifolds, cf. \cite{ChTW19}. 

As for the real Monge-Amp\`ere equation on a Riemannian manifold $(M,g)$ of dimension $n$, its {\it rough} version does not carry a substantial geometric meaning. It is nevertheless interesting from an analytic point of view. In this mentioned {\it rough} form
it can be written as
\[ \tag{1.15} \label{rma} \det \big( g + (\nabla^{LC})^2 \phi \big) = e^f \det g.\] In (\ref{rma}) the meaning of taking the determinant is that the symmetric bilinear forms $g + (\nabla^{LC})^2 \phi$ and $g$ are treated, via the metric, as the endomorphisms of
the tangent bundle and we take the determinant of this endomorphisms. These equations were treated for example in Li \cite{L90}, where they were solved under the non negative curvature assumption on $g$. This assumption was later removed by Urbas in \cite{U02}. 
Let us just mention that in case of this equation it is hard to obtain an easy normalization of $f$'s for which the equation can be solved.

The very special modification of this equation was treated earlier by Cheng and Yau in \cite{ChY82}. They considered affine manifolds, i.e. smooth real manifolds endowed with a flat torsion free connection, admitting a Riemannian metric, which they called affine K\"ahler metric, being locally given as the Hessian of a potential function in affine coordinates. The equation they considered is obtained by taking the affine K\"ahler metric in (\ref{rma}) and by exchanging the Levi--Civita connection there for the affine connection. 

\textbf{Acknowledgments:} The authors would like to use the opportunity of professor Sławomir Kołodziej's 60th birthday to express their gratitude to him for his constant support and advice. The first named author is supported by the National Science Center of Poland grant no. 
2017/26/E/ST1/00955. The second named author is supported by the National Science Center of Poland grant no. 2019/35/N/ST1/01372.

\section{Preliminaries}

In this section we introduce the notation and collect basic facts concerning hyperhermitian manifolds. We also prove technical or computational in nature results needed for the a priori estimates of Sections 3--5 in order to make the presentation there more straightforward.

\subsection{Hypercomplex geometry} 

We denote by
\[ \hh = \{ x_0+x_1\ii+x_2\jj+x_3\kk \: | \: x_0,x_1,x_2,x_3 \in \rr \} \] where $\ii^2= \jj^2= \kk^2=-1$ and $\ii \jj \kk = -1$ the field of quaternions with addition and multiplication being defined in the standard way. We consider $\hn$ as a right $\hh$ vector space. Let us recall what have become the standard definition. 
\begin{definition} \label{hypercomplex}
For a manifold $M$ of the real dimension $4n$ endowed with a triple of complex structures $I$, $J$, $K$ satisfying the quaternion relation
\[ I \circ J \circ K = -id_{TM} \] the tuple $(M,I,J,K)$ is called a hypercomplex structure. 
\end{definition}
A hypercomplex manifold admits many complex structures in particular the ones given 
\[ S_M= \{aI+bJ+cK \: | \: a^2+b^2+c^2=1 \} \]
(the so called twistor sphere).
\begin{remark}
We warn the reader that for us, in the whole text, endomorphisms act from the right on the tangent space. This convention is compatible with the one usually taken up in papers on hypercomplex geometry.
\end{remark} 
In that case each tangent space $T_x M$, for $x \in M$, becomes a right $\hh$-vector space where multiplication by $\ii$, $\jj$ and $\kk$ is given by $I_x$, $J_x$ and $K_x$ respectively.

Clearly the structure group of the hypercomplex manifold is reduced to $Gl_n(\hh)$ and vice versa each such reduction induces the almost complex structures $I$, $J$ and $K$ as in the definition above. The condition of the almost complex structures being integrable
is though not equivalent to the induced $Gl_n(\hh)$ structure being integrable in the strong sense of differential geometry, i.e locally $I$, $J$ and $K$ are not pull backs of the standard hypercomplex structure induced by $\ii$, $\jj$ and $\kk$ in $\hh^n$. In case the latter condition is satisfied such structures were studied in \cite{S75}. The latter is also equivalent to the existence of an atlas whose transition functions are affine maps with the endomorphism parts belonging to $Gl_n(\hh)$. On the bright side the integrability of the almost complex structures $I$, $J$ and $K$ implies the 0-integrability of the induced $Gl_n(\hh)$ structures (the reverse implication holds due to the Newlander-Nirenberg theorem), i.e. the existence of the $Gl_n(\hh)$ compatible torsion free connection. This is a non obvious result of Obata. Strong integrability of the $Gl_n(\hh)$ structure is equivalent to the Obata connection being in addition flat. 

\begin{theorem}{\cite{O56}}
For a hypercomplex manifold $(M,I,J,K)$ there exists a unique torsion free connection, denoted by $\nabla^{Ob}$, such that \[ \nabla^{Ob}I=\nabla^{Ob}J=\nabla^{Ob}K=0.\]
\end{theorem}

The coordinate expression for this connection can be found in \cite{O56}. An invariant global formula can be found for example in Gauduchon's paper \cite{G97b}. The existence of this connection will be very important for the technical results stated below and for the computations involved in deriving the a priori estimates for the equation (\ref{qma}). 
\begin{remark}
For a more detailed discussion on quaternionic geometry one can refer to the two excellent papers \cite{AM96} and \cite{S86}.
\end{remark}
When considering a hypercomplex manifold $(M,I,J,K)$ we obtain that 
\[ \hh \cong \{a id_{TM} + b I +cJ +dK \: | \: a,b,c,d \in \rr \} \subset End(M)\] acts from the right on $TM$ and from the left on $T^*M$ or more generally from the left on differential forms. The convention we use for the latter action is against the commonly used. Namely, given any field of endomorphisms $L$ on $TM$, acting according to our convention from the right, we define its left action on the space of complex valued smooth differential forms by 
\[ L: \Lambda^k_\cc (M) \ni \alpha \longmapsto \alpha(\cdot L,..., \cdot L) \in \Lambda^k_\cc (M).\]

\begin{remark}
From now on, whenever it happens that on a hypercomplex manifold $(M,I,J,K)$ we do not specify with respect to which complex structure the Hodge bidegree is taken, it is taken with respect to $I$.
\end{remark}

Let $(M,I,J,K)$ be a hypercomplex manifold. Let us remind that we have the Dolbeault operators \[ \partial :=\partial_I \text{ and } \bpar :=\overline{\partial_I}\] associated to the complex structure $I$ on $M$. We are going to introduce the quaternionic analogue of the $\overline{\partial}$ operator, or rather $d^c_I := I^{-1} \circ d \circ I$. In this we follow Verbitsky, cf. \cite{V02}, who defined the differential operator $\partial_J$ by 
\[ \tag{2.1} \label{tdo} \partial_J:=J^{-1} \circ \overline{\partial} \circ J. \] 
Since the operator $J$ acts on the complex forms by exchanging the bidegree components
\[ \tag{2.2} \label{jof} J: \Lambda^{p,q}_I(M) \longrightarrow \Lambda^{q,p}_I(M) \] 
the operator $\jpar$ acts on this forms by 
\[ \tag{2.3} \jpar: \Lambda^{p,q}_I(M) \rightarrow \Lambda^{p+1,q}_I(M).\]
We also introduce the operator $\bjpar$ defined formally again by twisting \[ \bjpar := J^{-1} \circ \partial \circ J, \] but as the operator $J$ is real it is equal to $\overline{(\jpar)}$ as well. 

It was observed by Verbitsky, \cite{V02,V07}, that the bicomplex \[ \Big(A^{p,q}:=\Lambda^{p+q,0}_I(M),\partial, \jpar \Big),\] called by him the quaternionic Dolbeault bicomplex, not only resembles the Dolbeault bicomplex 
\[ \Big(\Lambda^{p,q}, \partial, \bpar \Big) \] but it is also isomorphic to the so called Salamon complex, cf. \cite{S86}, introduced by Salamon in the broader context of quaternionic manifolds (see \cite{V07} for details).

For the further reference we would like to introduce the notion of $J$-realness. This is done as follows. The composition of the operator (\ref{jof}) with the bar operator is an involution on $\Lambda^{p,q}_I(M)$ if $p+q$ is even. The bundle of fixed points for this endomorphism in $\Lambda^{2k,0}_I (M)$ was denoted in \cite{V10} by $\Lambda^{2k,0}_{I,\rr} (M)$. In short, $\alpha \in \Lambda^{2k,0}_{I,\rr} (M) $ if and only if \[ J \alpha = \overline{\alpha} \] and such a from is called $J$-real. 

\subsection{Hyperhermitian metrics}

\begin{definition}
A Riemannian metric $g$ on a hypercomplex manifold $(M,I,J,K)$ is called  hyperhermitian if it is hermitian with respect to $I$, $J$ and $K$. For a hyperhermitian manifold $(M,I,J,K,g)$ and any $L \in S_M$ we denote the associated hermitian form by 
\[ \tag{2.4} \label{} \omega_L(X,Y)=g(X L,Y) \] for $X,Y \in \Gamma(TM)$. We define also the associated hyperhermitian form 
\[ \tag{2.5} \label{} \Omega = \omega_J - \ii \omega_K. \] 
\end{definition}

It is elementary to check that $\Omega \in \Lambda^{2,0}_{I,\rr}(M)$. Let us elaborate on the role of $\Omega$ in encoding the hyperhermitian metric 
$g$, cf. \cite{H90} Chapter 2, Lemma 2.72. Suppose $(V,g)$ is a right $\hh$-vector space with a hyperhermitian inner product $g$. Such inner products 
correspond bijectively to the hyperhermitian sesquilinear forms 
\[ \tag{2.6} \label{} H = g + \ii \omega_I + \jj \omega_J + \kk \omega_K . \] By 
introducing 
\[ \tag{2.7} \label{} \overline{h} = g + \ii \omega_I\] and $\Omega$ as above we obtain 
\[ \tag{2.8} \label{} H = \bar{h} + \jj \Omega.\] What is more 
\[ \tag{2.9} \label{} \Omega = \overline{h}(\cdot \jj, \cdot ).\] Consequently $\bar{h}$, and in turn also $H$, is completely determined by $\Omega$ satisfying for any $v,w \in V$ 
\[\Omega(v  \jj ,w \jj) = \overline{\Omega(v,w)},\] or after taking the complexification of $V$, 
\[\label{posit1} \tag{2.10} \Omega(\cdot \jj, \cdot \jj ) = \overline{\Omega}\] and \[\Omega(v,v \jj ) = \overline{h}( v \jj, v \jj ) \geq 0,\]  or after taking the complexification, 
\[ \label{posit2} \tag{2.11} \Omega(z,\bar{z} \jj ) = \overline{h}( z \jj, \bar{z} \jj ) \geq 0 \] 
for any $z \in V^{1,0}_\ii$. The conditions (\ref{posit1}) and (\ref{posit2}) give the meaning of the inequality in (\ref{qma}).
\begin{remark}
In calculations it will be customary to assume that at the point of interest the hyperhermitian structure $(T_x M, I_x, J_x,K_x,g_x)$ is isomorphic to the standard model below, cf. (\ref{standj}) and (\ref{diag}). 
This can trivially be seen to be possible by taking the orthonormal basis of the form $e_0$, $e_0I$, $e_0J$, $e_0K$, ..., $e_{n-1}$, $e_{n-1}I$, $e_{n-1}J$, $e_{n-1}K$ for $(T_x M, I_x, J_x,K_x,g_x)$. As was discussed in \cite{Sr19}
in the presence of a second hyperhermitian metric the first one may still be assumed to be standard while the second one being diagonal. \end{remark} 

The right multiplications by $\ii$, $\jj$, and $\kk$ act on $\hn$ defining the almost complex structures $I$, $J$ and $K$ respectively. Let us introduce, next to the real coordinates (\ref{rc}), the holomorphic coordinates, for the complex structure $I$, by decomposing 
\[ \label{cc} \tag{2.12} q_i=z_{2i}+\jj z_{2i+1} \] for $i \in \{0,...,n-1\}$. As an easy calculation shows the action of $J$ in this holomorphic coordinates is
\[ \label{standj} \tag{2.13} \begin{gathered} (\partial_{z_{j}})J=(-1)^{j}\partial_{\overline{z_{j+(-1)^j}}}, \\ J^{-1}(d z_{j})=(-1)^{j}d\overline{z_{j+(-1)^j}} \end{gathered} \] for $j \in \{0,...,2n-1\}$.
Take a standard inner product on $\hn$, in coordinates from (\ref{rc}),
\[ g = dx_{4i} \otimes dx_{4i} + dx_{4i+1} \otimes dx_{4i+1}+ dx_{4i+2} \otimes dx_{4i+2} +dx_{4i+3} \otimes dx_{4i+3}. \]
We easily get the following expressions for the quantities associated with this hyperhermitian structure $(\hn, I,J,K,g)$
\[ \tag{2.14} \label{diag} \begin{gathered} \omega_I=-dx_{4i+1} \otimes dx_{4i } + dx_{4i} \otimes dx_{4i+1} + dx_{4i+3} \otimes dx_{4i+2} - dx_{4i+2} \otimes dx_{4i+3} \\
= dx_{4i} \wedge dx_{4i+1} + dx_{4i+3} \wedge dx_{4i+2}=\frac{\ii}{2}(dz_{2i} \wedge d\overline{z_{2i}} + dz_{2i+1} \wedge d \overline{z_{2i+1}}),\\
\omega_J=dx_{4i} \wedge dx_{4i+2} + dx_{4i+1} \wedge dx_{4i+3},\\
\omega_K=dx_{4i+2} \wedge dx_{4i+1} + dx_{4i} \wedge dx_{4i+3},\\
\Omega =\omega_J - \ii \omega_K =dx_{4i} \wedge dx_{4i+2} + dx_{4i+1} \wedge dx_{4i+3} - \ii dx_{4i+2} \wedge dx_{4i+1} -\ii dx_{4i} \wedge dx_{4i+3}\\
= (dx_{4i} + \ii dx_{4i+1})\wedge(dx_{4i+2}-\ii dx_{4i+3})=dz_{2i} \wedge dz_{2i+1}. \end{gathered}\]
\begin{remark} \label{ccvrc} One should note that the real coordinates introduced in (\ref{rc}) are not given by taking the real and imaginary part decomposition of the complex coordinates (\ref{cc}). The relation between these coordinates is \[ \begin{gathered} z_{2j}=x_{4j}+ x_{4j+1} \ii, \\ z_{2j+1}= x_{4j+2} - x_{4j+3} \ii, \end{gathered} \] for $j=0,...,n-1$.  
\end{remark}

Before introducing certain classes of hyperhermitian metric let us recall that in the paper \cite{G97} Gauduchon has distinguished in the affine space of all the hermitian connections, i.e. those satisfying \[ \nabla I = \nabla g = 0\] for a hermitian manifold $(M,I,g)$, the affine line of the canonical connections. Among them two classical ones will be important for our presentation.
\begin{proposition}{\cite{G97}} Let $(M,I,g)$ be a hermitian manifold. There exists a unique hermitian connection, denoted by $\nabla^{Ch}_{I,g}$, which will be called the Chern connection characterized by the fact that \[\big(\nabla^{Ch}_{I,g} \big)^{0,1} = \bpar. \]
There exists as well a unique hermitian connection, denoted by $\nabla^{B}_{I,g}$, which will be called the Bismut connection characterized by the fact that its torsion tensor, after lowering the upper index by $g$, is a three form.
\end{proposition}
Coming back to the hyperhermitian metric on hypercomplex manifolds, when considering the above connections associated to the hermitian structure $(M,L,g)$ for any $L \in S_M$ we add a subscript $L$, eg. $\nabla^{Ch}_{L,g}$. Let us recall the following definition:
\begin{definition}
A hyperhermitian metric $g$ on $(M,I,J,K)$ is called hyperK\"ahler, HK for short, if any of the equivalent conditions are satisfied \begin{itemize} \item $d\omega_I=d \omega_J = d \omega_K=0$, \item $d \Omega = 0$, \item $\nabla^{Ob} = \nabla^{LC}$, \item $\nabla^{B}_{I,g}= \nabla^{B}_{J,g}=\nabla^{B}_{K,g}= \nabla^{LC}$. \end{itemize} If moreover $M$ is simply connected the conditions above are equivalent to \[ Hol(g) \subset Sp(m)\] and $I,J,K$ being induced by this holonomy group.
\end{definition}

This class of metrics is standard to consider from the point of view of Berger's Riemannian holonomy theorem. Indeed it implies that $Sp(n)$, and $Sp(n)\cdot Sp(1)$, corresponding respectively to the hyperK\"ahler and quaternionic K\"ahler metrics,
are the only infinite families occurring, cf. \cite{B87}, which correspond to the hypercomplex, respectively quaternionic, geometry. There are only two known deformation classes in each dimension and two isolated examples, due to O'Grady, of hyperK\"ahler manifolds. Partially because of that one may be tempted to look for a natural generalizations of those. A possible attempt is as follows.  
\begin{definition}
A hyperhermitian manifold $(M,I,J,K,g)$ is called HKT, which stands for hyperK\"ahler with torsion, if any of the equivalent conditions is satisfied \begin{itemize} \item $\partial \Omega=0$, \item $\nabla^{B}_{I,g} = \nabla^{B}_{J,g}=\nabla^{B}_{K,g}$.\end{itemize}
\end{definition}
From the second condition it follows that HKT structures are natural differential geometric generalizations of HK structures where the torsion just vanishes. These metrics emerged originally from mathematical physics. More exactly connections with special holonomy and skew torsion occur naturally while studying the target space of sigma models in quantum theory. In the presence of the so called Wess-Zumino term and supersymmetry the HKT metrics appear, cf. \cite{HP96}.

An established mathematical treatment of basic properties of HKT manifolds is \cite{GP00}. One should note though that despite the name HKT manifolds do not admit  K\"ahler metrics in general. In fact a result of Verbitsky \cite{V05} shows they can be K\"ahler only in the case when the manifold already admits
HK metric.

\subsection{Quaternionic Monge-Amp\`ere equation}

The quaternionic Monge-Amp\`ere equation (\ref{qma}) on HKT manifolds proposed by Alesker and Verbitsky, cf. \cite{AV10}, naturally solves the prescribed trivialization problem (\ref{cplexvol})-(\ref{pcv}). The authors suggested to look for an HKT metric whose associated HKT form is \[ \Omega_\phi := \Omega + \partial \partial_J \phi \] for some smooth real function $\phi$ and for which $\Omega_\phi^n$ is the section we want to obtain because the $\pa \jpar \phi$ perturbation preserves the HKT condition. If such a $\phi$ exists then this new HKT metric $g_\phi$ can be obtained from $\Omega_\phi$ by applying the reasoning from the section above. We denote the associated hermitian forms by adding a subscript $g_\phi$, eg. $\omega_{I,g_\phi}$.

As we noted the canonical bundle $K_I(M)$ of a given HKT manifold is trivial topologically, with $\Omega^n$ providing a smooth trivialization. Unlike in the K\"ahler case where, up to the finite covering, topological triviality gives the holomorphic one,
here the canonical bundle is not holomorphically trivial in many cases, Hopf surfaces being one example. Arguably, cf. \cite{V09}, HKT metrics for which $\Omega^n$ is holomorphic trivialization constitute a hypercomplex analogue of a Calabi--Yau manifold. 
These metrics have in particular the property of being balanced with respect to any complex structure from $S_M$. They perfectly fit into the recently active stream of research on generalizations of Calabi-Yau spaces, the so called torsion
Calabi--Yau manifolds, cf. \cite{T15, Pi19}.

Let us now turn to other consequences of Theorem \ref{mainthm} advocated in the introduction. In doing so let us also keep in mind that actually an expectation is that Conjecture \ref{qCc} should hold even after dropping the assumption
on the holomorphic triviality of the canonical bundle $K_I(M)$.

For a hyperhermitian manifold $(M,I,J,K,g)$ and the constant $c_n$ depending only on the dimension, as can be seen from (\ref{diag}),  \[\Omega^n \wedge \overline{\Omega}^n=c_n \omega_I^{2n}.\] Consequently we see that the quaternionic Monge-Amp\`ere equation (\ref{qma}) is solvable for any $f$ if and only if \[(\omega_{I, g_\phi})^{2n} = \frac 1 c_n (\Omega + \partial \partial_J \phi)^n \wedge (\overline{\Omega + \partial \partial_J \phi})^n = \frac 1 c_n e^{(2f+2b)}\Omega^n \wedge \overline{\Omega}^n = e^{F+C}\omega_I^{2n} \] is satisfied with suitable $\phi$ and $C$ for any $F$. This shows that Remark \ref{vol} follows from Theorem \ref{mainthm}.

The above means that any representative of $c_1^{BC}(M,I)$ can be obtained as the Chern--Ricci curvature of an HKT metric $g_\phi$ as the standard calculation shows. This was noted already in \cite{Ma11} for what Madsen calls the projected Chern form,
cf. Section 7.1.3 in \cite{Ma11}.
Indeed, as is well known prescribing the Chern-Ricci curvature is equivalent to solving the Monge-Amp\`ere equation 
\[\left( \det ({g_{\phi}}_{i \overline{j}})_{i,j}\right)= e^{F+b}\left( \det (g_{i\bar{j}})_{i,j}\right)\] 
for some $b \in \rr$. This in turn means, by going from the chart expression to the global one, that 
\[ \omega_{I,g_\phi}^n = e^{F+b} \omega^n_I \] justifying that Proposition \ref{pbc} follows from Remark \ref{vol}.

\subsection{Technical results}

From now on we assume that the components of all tensors are taken with respect to a holomorphic coordinates $z_i$ for $I$. Another important convention we use is that whenever an unknown, eg. $i$, appears as an index its range is in $\{0,...,2n-1\}$. When an expressions $2i$ or $2i+1$ involving an unknown appears as an index the range for $i$ is in $\{0,...,n-1\}$. We often omit the summation symbols when it is clear that the summation takes place even when the Einstein summation convention does not apply directly.
    
First of all let us note the following properties of the operators $\pa$, $\jpar$, $\bpar$, $\bjpar$ introduced in this section. This result is elementary and based only on the integrability and anti commutativity of $I$ and $J$ but we do not know a reference containing the proof.
\begin{lemma}\label{ac}
For a hypercomplex manifold $(M,I,J,K)$ the following holds 

\[ \label{} \tag{2.15} \pa^2 = \bpar^2=\jpar^2=\bjpar^2=0, \]

\[\label{} \tag{2.16} \partial \bpar + \bpar \pa = \jpar \bjpar + \bjpar \jpar = \pa \jpar + \jpar \pa = \bpar \bjpar + \bjpar \bpar = \jpar \bpar + \bpar \jpar = \bjpar \pa + \pa \bjpar = 0. \]
\end{lemma}
\begin{proof} One can simply use the facts that 
\[ \tag{2.17} \label{*} dd^c_J+d^c_Jd=0, \] 
\[ \tag{2.18} \label{**} dd^c_K + d^c_K d = 0. \] 
This follows from the integrability of $J$ and $K$. Then by rewriting $d$ as $\partial + \bpar$ we obtain
\[ d = \partial + \bpar , \] 
\[ d^c_I = \ii (\bpar - \partial ) , \]
\[ d^c_J = J^{-1} \circ (\partial + \bpar ) \circ J = \bjpar + \jpar , \]
\[ d^c_K = (IJ)^{-1} \circ d \circ (IJ) = J^{-1} \circ (I^{-1} \circ d \circ I) \circ J = J^{-1} \circ d^c_I \circ J = \ii ( \jpar - \bjpar ). \]  
Comparing both sides in (\ref{*}) and (\ref{**}) and taking into an account the Hodge bidegrees with respect to $I$ gives the claim. More precisely (\ref{*}) gives
\[ \tag{2.19} \label{A} \partial \bjpar + \partial \jpar + \bpar \bjpar + \bpar \jpar + \bjpar \partial + \jpar \partial + \bjpar \bpar + \jpar \bpar =0, \] while (\ref{**}) gives
\[ \tag{2.20} \label{B} \partial \jpar - \partial \bjpar + \bpar \jpar - \bpar \bjpar + \jpar \partial - \bjpar \partial + \jpar \bpar - \bjpar \bpar =0.\] It turns out that 
\[ \partial \jpar + \jpar \partial = 0, \] 
as it is the component of bidegree $(2,0)$ of the left hand side of (\ref{A}),
\[ \bpar \bjpar + \bjpar \bpar = 0, \]
as it is of bidegree $(0,2)$ of (\ref{A}), 
\[ \tag{2.21} \label{C} \partial \bjpar + \bpar \jpar + \bjpar \partial  + \jpar \bpar = 0, \]
as it is of bidegree $(1,1)$ of (\ref{A}),
\[ \tag{2.22} \label{D} - \partial \bjpar + \bpar \jpar - \bjpar \partial + \jpar \bpar = 0, \] 
as it is of bidegree $(1,1)$ of (\ref{B}).
Adding and subtracting (\ref{C}) and (\ref{D}) we obtain 
\[ \bpar \jpar + \jpar \bpar = 0 \]
and
\[ \partial \bjpar + \bjpar \partial = 0. \]  
The only two remaining identities
\[\partial \bpar + \bpar \partial = 0, \]
\[ \jpar \bjpar + \bjpar \jpar = J^{-1} \circ ( \bpar \partial + \partial \bpar ) \circ J = 0\]
of course do hold.
\end{proof}
The first conclusion we may draw from this is, what we will use constantly during the computations of the sections to follow, that certain identities involving derivatives of the components of $J$ vanish locally and not only at a fixed point.
\begin{remark}
For any holomorphic coordinates for $I$ we have
\[ 0 = \big(\pa \jpar + \jpar \pa \big) (z_i) = \pa J^{-1} \bpar z_i + J^{-1} \bpar J dz_i,\] 
consequently 
\[ \bpar J \big(dz_i \big) = 0. \] In coordinates this reads \[\label{ahv} \tag{2.23} \bpar \big( J_{\overline{k}}^i d\overline{z_k} \big) = J_{\bar{k},\bar{l}}^i d\overline{z_l} \wedge d\overline{z_k} = \sum_{l<k} \big( J_{\bar{k},\bar{l}}^i - J_{\bar{l},\bar{k}}^i \big) d\overline{z_l} \wedge d\overline{z_k} =0. \] 
This in turn provides 
\[ J_{\bar{k},\bar{l}}^i = J_{\bar{l},\bar{k}}^i \] 
and by conjugation 
\[ J_{k,l}^{\bar{i}} = J_{l,k}^{\bar{i}} \] for all $i$, $j$ and $k$.
\end{remark}
In order to have even better control of the derivatives of $J$ we have to stick to the point. Let $\ob$ be the Obata connection for $(M,I,J,K)$. Since it is the complex, in particular for $I$, torsion free connection we have, in any holomorphic chart for $I$,
\[ \tag{2.24} \label{} \ob_{\partial_{z_i}} \partial_{\overline{z_j}}=\ob_{\partial_{\overline{z_i}}} \partial_{{z_j}}=0,\]
\[ \tag{2.25} \label{} \begin{gathered} \ob_{ \partial_{z_i}} \partial_{z_j} = \Gamma_{ij}^k \partial_{z_k}, \\ 
\ob_{ \partial_{\overline{z_i}} } \partial_{\overline{z_j}} = \Gamma_{\overline{i} \overline{j}}^{\overline{k}} \partial_{\overline{z_k}},\end{gathered} \]
\[ \tag{2.26} \label{} \begin{gathered} \ob_{ \pa_{z_i}} {dz_j} = -\Gamma_{im}^j d{z_m}, \\
\ob_{ \pa_{\overline{z_i}}} {\overline{dz_j}} = -\Gamma_{\bar{i}\bar{m}}^{\bar{j}} d{\overline{z_m}}. \end{gathered}\]
The condition \[\ob J = 0 \] gives 
\[ 0 
= \ob_{\partial_{z_i}} \big( J_{\overline{k}}^l d \overline{z_k} \otimes \partial_{z_l} + J_{k}^{\overline{l}} d z_k \otimes \partial_{\overline{z_l}} \big) \]
\[
= J_{\overline{k},i}^l d \overline{z_k} \otimes \partial_{z_l} + J_{\bar{k}}^l d \overline{z_k} \otimes \big( \Gamma_{il}^m \partial_{z_m}\big) + J_{k,i}^{\overline{l}} d z_k \otimes \partial_{\overline{z_l}} + J_{k}^{\overline{l}} \big( -\Gamma_{im}^k d{z_m} \big) \otimes \partial_{\overline{z_l}}.\] 
This in turn allows us to obtain the expressions 
\[ \label{chrj} \tag{2.27} \begin{gathered}  J_{\overline{k},i}^l = -J_{\bar{k}}^m \Gamma_{im}^l, \\
J_{k,i}^{\overline{l}} = J_{m}^{\overline{l}} \Gamma_{ik}^m,\\
J_{k,\bar{i}}^{\bar{l}} = -J_{k}^{\bar{m}} \Gamma_{\overline{i} \overline{m}}^{\bar{l}}, \\
J_{\bar{k},\bar{i}}^{l} = J_{\bar{m}}^{l} \Gamma_{\overline{i} \overline{k}}^{\bar{m}}.\end{gathered} \]
Since the Obata connection is torsion free and $I$-complex, for any chosen $p \in M$ we can choose $I$-holomorphic, geodesic coordinates which, from (\ref{chrj}), gives at $p$ the equalities
\[ \tag{2.28} \label{geodesic} J_{\overline{k},i}^l = J_{k,i}^{\overline{l}} = J_{k,\bar{i}}^{\bar{l}}=J_{\bar{k},\bar{i}}^{l}=0.\]
Finally, let us find the expression for the $\pa \jpar \phi$ perturbation at the point, in local coordinates.
\begin{lemma}
At the point on $(M,I,J,K)$, in any coordinates satisfying (\ref{standj}), for any $\phi$ we have:
\[ \tag{2.29} \label{} \begin{gathered} \partial_J \phi=(J^{-1} \bpar J) \phi=J^{-1} \big( \bpar \phi \big) = J^{-1} \Big(\sum\limits_{j=0}^{2n-1}  \phi_{\overline{j}} d\overline{z_j} \Big) \\
=\sum\limits_{j=0}^{2n-1} \phi_{\overline{j}} J^{-1}(d \overline{z_j})=\sum\limits_{j=0}^{2n-1} \phi_{\overline{j}} (-1)^j d {z_{j+(-1)^j}}= \sum\limits_{j} \big( \phi_{\overline{2j}} d {{z_{2j+1}}} - \phi_{\overline{2j+1}} d {{z_{2j}}} \big), \end{gathered}\]

\[ \begin{gathered} \pa \jpar \phi = \pa \big( \phi_{\overline{j}} J^{-1} d \overline{z_j} \big) = \phi_{i \overline{j}} dz_i \wedge J^{-1} d \overline{z_j}=\sum\limits_{i,j} \left( (-1)^{j+1} \phi_{i \overline{j+(-1)^j}} \right) dz_i \wedge dz_j \\
= \sum\limits_{i,j} \big( \phi_{2i \overline{2j}} dz_{2i} \wedge d z_{2j+1} + \phi_{2i+1 \overline{2j}} dz_{2i+1} \wedge d z_{2j+1} - \phi_{2i \overline{2j+1}} dz_{2i} \wedge d z_{2j} - \phi_{2i+1 \overline{2j+1}} dz_{2i+1} \wedge d z_{2j} \big). \end{gathered}\]
The last one, after rearrangement, gives
\[ \label{ghess} \tag{2.30} \begin{gathered} \pa \jpar \phi = \sum\limits_{i,j} \big( \phi_{2i \overline{2j}} + \phi_{2j+1 \overline{2i+1}} \big) dz_{2i} \wedge d z_{2j+1} \\
+ \sum\limits_{i<j} \big( \phi_{2i+1 \overline{2j}} - \phi_{2j+1 \overline{2i}} \big) dz_{2i+1} \wedge d z_{2j+1} \\
+ \sum\limits_{i<j} \big( \phi_{2j \overline{2i+1}} - \phi_{2i \overline{2j+1}} \big) dz_{2i} \wedge d z_{2j}. \end{gathered} \]

In the coordinates in which $\pa \jpar \phi$ is diagonal, formula (\ref{ghess}) reads 
\[ \label{} \tag{2.31} \phi_{2i \overline{2j}} =- \phi_{2j+1 \overline{2i+1}}\] 
for $i \not = j$ and 
\[ \tag{2.32} \label{} \begin{gathered} \phi_{2i+1 \overline{2j}} = \phi_{2j+1 \overline{2i}} \\
\phi_{2j \overline{2i+1}} = \phi_{2i \overline{2j+1}} \end{gathered} \]
for any $i,j$.
\end{lemma}
Using the above discussion we provide, for later reference, the expression for the Chern laplacian involving the quaternionic Hessian.
\begin{proposition} \label{chernlapl}
Let $(M,I,J,K,g)$ be a hyperhermitian manifold and $\phi \in C^\infty(M)$, then
\[ 2n \frac{\pa \jpar \phi \wedge \Omega^{n-1}}{\Omega^n} = \Delta^{Ch}_{I,g} \phi.\]
\end{proposition}  
\begin{proof}
It is well known that the Chern laplacian can be expressed as
\[2n \frac{\pa \bpar \phi \wedge \omega_I^{2n-1}}{\omega_I^{2n}}.\]
Let us choose any holomorphic coordinates such that at the point $x \in M$ the hyperhermitian structure is standard, in the sense that $(2.13)$ and $(2.14)$ are satisfied. In those coordinates we see, since $\omega_I=\frac{\ii}{2}(dz_{2i} \wedge d\overline{z_{2i}} + dz_{2i+1} \wedge d \overline{z_{2i+1}})$, 
\[2n \frac{\pa \bpar \phi \wedge \omega_I^{2n-1}}{\omega_I^{2n}} = 2(\phi_{2i\overline{2i}} + \phi_{2i+1\overline{2i+1}}).\]
On the other hand, because of $(2.30)$ and $\Omega = dz_{2i} \wedge dz_{2i+1}$, we see that also
\[2n \frac{\pa \jpar \phi \wedge \Omega^{n-1}}{\Omega^n} = 2(\phi_{2i\overline{2i}} + \phi_{2i+1\overline{2i+1}})\]
as required.
\end{proof}
We recall the basic facts concerning the Pfaffian. The following proposition is probably well known but we do not know the reference. The proof reduces to defining Pfaffian as below and checking the claimed equality (\ref{detpf}) of polynomials on sufficiently many skew-symmetric matrices.
\begin{proposition} \label{Pf}
There exists a polynomial which we denote by $\Pf$, with real coefficients, of degree $n$ on the space of skew-symmetric complex matrices of size $2n$, i.e. those satisfying $A^T=-A$, such that \[\label{detpf} \tag{2.33} \det = {\Pf}^2 \] as polynomials on this space.
\end{proposition}

The polynomial $\Pf$ from Proposition \ref{Pf}, is defined only up to the sign. We make the following choice.

\begin{definition} \label{Pfd} For a skew-symmetric complex matrix \[ M = \left( m_{ij} \right)_{i,j =1,...,2n} \] we define the Pfaffian of $M$ as \[ \label{} \tag{2.34} \Pf(M) e_1 \wedge ... \wedge e_{2n} = \frac 1 {n!}\left( \sum\limits_{i<j}m_{ij}e_i \wedge e_j \right)^n\] where $e_i$ is the canonical basis of $\cc^{2n}$. 
\end{definition}

With this definition, the mentioned convention for tensors components and notation (\ref{of}), it follows immediately that writing the equation (\ref{qma}) in holomorphic coordinates gives the equation (\ref{cpe}). Note that the original equation (\ref{qma}) could have been rewritten in that way only because the associated hyperhermitian forms are of the Hodge type $(2,0)$. 

In deriving the estimates we will differentiate equation (\ref{cpe}) and in order to do that we would like to know the formula for the derivative of the Pfaffian. As we will be concerned only with the matrices with positive Pfaffian this can be derived from the more familiar formula for determinant derivatives coupled with Proposition \ref{Pf}. The result is as follows. 

\begin{lemma} Suppose $A = \big[ A_{ij}\big]_{i,j}$ is a complex skew-symmetric $2n \times 2n$ matrix with positive Pfaffian depending on variables $t$ and $s$. Its derivatives are given by   
\[\label{pfder} \tag{2.35} \frac{\pa}{\pa t} \Big( \log \Pf (A) \Big) = \frac{1}{2} \tr \Big( A^{-1} \frac{\pa}{\pa t} A \Big),  \]
\[\label{pfderder} \tag{2.36} \frac{\pa}{\pa s} \frac{\pa}{\pa t} \Big( \log \Pf (A) \Big) = \frac{1}{2} \bigg[ \tr \Big( A^{-1} \big( \frac{\pa}{\pa s} \frac{\pa}{\pa t} A\big) \Big) - \tr \Big( A^{-1} \big( \frac{\pa}{\pa s} A \big) A^{-1} \big( \frac{\pa}{\pa t} A\big) \Big) \bigg].\]
\end{lemma}

\begin{remark}
From now on we assume that at the point of interest the holomorphic coordinates are chosen so that the hyperhermitian structure $(I,J,K,g)$ is standard, i.e. (\ref{standj}) and (\ref{diag}) hold, the form $\Omega_\phi:= \Omega + \pa \jpar \phi$ is diagonal, or equivalently the metric $g_\phi$ is and that (\ref{geodesic}) holds. Such an arrangement is possible because after choosing the geodesic coordinates we may make a linear change of those gaining the simultaneous diagonalization of both metrics. The vanishing of the Christoffel symbols is preserved under linear change of coordinates. From time to time we may refer to such a coordinates as canonical. 
\end{remark}

\section{$C^1$ estimate}
The goal of this section is to prove the following $C^1$ a priori estimate:
\begin{theorem}\label{c1b}
Let $(M,I,J,K,g)$ be a compact, connected hyperK\"ahler manifold. There exists a constant $C$ depending on $f$, $\sup_M |\phi|$ and the hyperhermitian structure $(I,J,K,g)$ such that for any solution $\phi$ of the equation (\ref{qma}) the estimate \[ \label{c1e} \tag{3.1} |d \phi|_g \leq C \] holds.
\end{theorem}
\begin{proof}
Let us define $\beta$ by \[ \label{grad} \tag{3.2} \beta \Omega^n = n \partial \phi \wedge \jpar \phi \wedge \Omega^{n-1}.\] It is easy to see, for example by rewriting this in canonical coordinates, that \[ \beta = \frac 1 4 |d\phi|_g^2 .\]
Since, after taking the logarithm, the linearization of the equation (\ref{qma}) is, up to the constant, the Chern Laplacian with respect to the hermitian structure $(I,g_\phi)$ on $M$, we note the useful form-type formula for this, cf. Proposition \ref{chernlapl}, in our setting \[ \tag{3.3} \label{} \pa \jpar f \wedge \Omega^{n-1}_\phi = \Big( \frac{1}{2n} \Delta^{Ch}_{I,g_\phi} f \Big)  \Omega^n_\phi.\]

Following Błocki, cf. \cite{B09}, we consider the quantity \[ \label{} \tag{3.4} \alpha = \log \beta - \gamma \circ \phi\] for a function $\gamma : \rr \rightarrow \rr$ to be specified below. 

All the computations from now on will be curried out at a maximum point of $\alpha$. As the operators $\pa$ and $\jpar$ are of pure first order we note that \[ \label{extrem} \tag{3.5} \pa \alpha = \frac{\pa \beta}{\beta} - \gamma' \pa \phi = 0, \]
\[\label{extrem'} \tag{3.6} \pa_J \alpha = \frac{\pa_J \beta}{\beta} - \gamma' \pa_J \phi = 0. \]

Furthermore
\[  \label{} \tag{3.7} \begin{gathered} \pa \jpar \alpha = \frac{\pa \jpar \beta}{\beta} - \frac{\pa \beta \wedge \jpar \beta}{\beta^2} - \gamma'' \pa \phi \wedge \jpar \phi - \gamma'\pa \jpar \phi \\ 
= \frac{\pa \jpar \beta}{\beta} - \frac{\pa \beta \wedge \jpar \beta}{\beta^2} - \gamma'' \pa \phi \wedge \jpar \phi - \gamma' \big(\Omega+ \pa \jpar \phi \big) + \gamma' \Omega \\
= \frac{\pa \jpar \beta}{\beta} - \big( (\gamma')^2 + \gamma'' \big) \pa \phi \wedge \jpar \phi - \gamma' \big(\Omega+ \pa \jpar \phi \big) + \gamma' \Omega. \end{gathered} \]

Taking the bar of (\ref{grad}) results in
\[\beta \overline{\Omega}^n = n \bpar \phi \wedge \bjpar \phi \wedge \overline{\Omega}^{n-1}.\]
Next, taking $\jpar$ of both sides, because of the hyperK\"ahler assumption, we get
\[\jpar \beta \wedge \overline{\Omega}^n = n \jpar \bpar \phi \wedge \bjpar \phi \wedge \overline{\Omega}^{n-1} - n \bpar \phi \wedge \jpar \bjpar \phi \wedge \overline{\Omega}^{n-1}\]
and by taking $\pa$, from the same reason as above, we end up with
\[\pa \jpar \beta \wedge \overline{\Omega}^n \]
\[ \label{} \tag{3.8} = n \pa \jpar \bpar \phi \wedge \bjpar \phi \wedge \overline{\Omega}^{n-1} + n \jpar \bpar \phi \wedge \pa \bjpar \phi \wedge \overline{\Omega}^{n-1} \]
\[ - n \pa \bpar \phi \wedge \jpar \bjpar \phi \wedge \overline{\Omega}^{n-1} + n \bpar \phi \wedge \pa \jpar \bjpar \phi \wedge \overline{\Omega}^{n-1}.\]

From the equation (\ref{qma})
\[\big( \Omega + \partial \jpar \phi \big)^n = e^f \Omega^n\]
by taking $\bpar$ and applying Lemma \ref{ac} we obtain 
\[n \pa \jpar \bpar \phi \wedge \Omega_\phi^{n-1} = \bpar e^f \wedge \Omega^n\]
while by taking $\bjpar$
\[n \pa \jpar \bjpar \phi \wedge \Omega_\phi^{n-1} = \bjpar e^f \wedge \Omega^n.\]
From this we obtain 
\[\pa \jpar \beta \wedge \Omega_\phi^{n-1}\wedge \bOmega^n \]
\[ \label{fourthgrad} \tag{3.9} = - \bjpar \phi \wedge \bpar e^f \wedge \Omega^n \wedge \bOmega^{n-1} + n \jpar \bpar \phi \wedge \pa 
\bjpar \phi \wedge \Omega_\phi^{n-1} \wedge  \overline{\Omega}^{n-1} \]
\[- n \pa \bpar \phi \wedge \jpar \bjpar \phi \wedge \Omega_\phi^{n-1} \wedge \overline{\Omega}^{n-1} + \bpar \phi \wedge \bjpar 
e^f \wedge \Omega^n \wedge \bOmega^{n-1}. \]

We now turn to evaluating the required quantities of second order present in the expression from (\ref{fourthgrad}). They are equal to

\[ \label{secondgrad1} \tag{3.10} \jpar \bpar \phi = \jpar \left(\phi_{\bar{j}} d\overline{z_j}\right) = J^{-1} \bpar \left(\phi_{\bar{j}} J d\overline{z_j}\right) = J^{-1} \left( \phi_{\bar{j} \bar{k}} d\overline{z_k} \wedge J d\overline{z_j}\right) = \phi_{\bar{j}\bar{k}} J^{-1}d\overline{z_k} \wedge d\overline{z_j},\]

\[ \label{secondgrad2} \tag{3.11} \pa \bjpar \phi = \pa J^{-1} \pa \phi = \pa J^{-1} \left( \phi_i d z_i \right) = \pa \left(\phi_i J^{-1} dz_i \right)= \phi_{ij} dz_j \wedge J^{-1} d z_i,\]

\[\label{secondgrad3} \tag{3.12} \pa \bpar \phi = \phi_{i \bar{j}} dz_i \wedge d\overline{z_j},\]

\[ \label{secondgrad4} \tag{3.13} \jpar \bjpar \phi = J^{-1} \bpar J J^{-1} \pa \phi = J^{-1} \bpar \pa \phi = \phi_{i \bar{j}} J^{-1}(d\overline{z_j}) \wedge J^{-1} dz_i.\]

At a maximum point of $\alpha$ we have
\[ \label{chernlaplaciangrad} \tag{3.14} \begin{gathered} 0 \geq \frac{1}{2n} \Delta^{Ch}_{I,g_\phi} \alpha = \frac{\pa \jpar \alpha \wedge \Omega_\phi^{n-1} \wedge \bOmega^n}{\Omega_\phi^n\wedge \bOmega^n} \\
 = \frac{\pa \jpar \beta \wedge \Omega_\phi^{n-1} \wedge \bOmega^n}{ \beta \Omega_\phi^n\wedge \bOmega^n}  - \big( (\gamma')^2 + \gamma'' \big) \frac{ \pa \phi \wedge \jpar \phi \wedge \Omega_\phi^{n-1} \wedge \bOmega^n}{ \Omega_\phi^n\wedge \bOmega^n} \\
 - \gamma'\frac{ \Omega_\phi^{n} \wedge \bOmega^n}{  \Omega_\phi^n\wedge \bOmega^n} + \gamma' \frac{\Omega \wedge \Omega_\phi^{n-1} \wedge \bOmega^n}{  \Omega_\phi^n\wedge \bOmega^n}. \end{gathered} \]

Applying (\ref{secondgrad1}) -- (\ref{secondgrad4}) in (\ref{fourthgrad}) and rewriting in coordinates gives us that those quantities are equal to

\[ \label{gradderdersec} \tag{3.15} \begin{gathered} \frac{\pa \jpar \beta \wedge \Omega_\phi^{n-1} \wedge \bOmega^n}{ \beta \Omega_\phi^n \wedge \bOmega^n} \\
 = \frac{1}{n} \frac{1}{ \beta} \Big( \frac{\phi_{\bar{i}} (e^f)_i}{e^f} +  \frac{ \phi_{i} (e^f)_{\bar{i}}}{e^f} + \frac{\phi_{\bar{j}\bar{k}} \phi_{jk}}{|\Omega^\phi_{k k+(-1)^k}|} + \frac{\phi_{i\bar{j}} \phi_{ \bar{i}j}}{|\Omega^\phi_{i i+(-1)^i}|} \Big), \end{gathered} \]

\[ \label{hmmm} \tag{3.16} - \big( (\gamma')^2 + \gamma'' \big) \frac{ \pa \phi \wedge \jpar \phi \wedge \Omega_\phi^{n-1} \wedge \bOmega^n}{ \Omega_\phi^n\wedge \bOmega^n}= - \frac 1 n \big( (\gamma')^2 + \gamma'' \big) \frac{\phi_i \phi_{\bar{i}}}{|\Omega^\phi_{i i+(-1)^i}|},\]

\[ \label{} \tag{3.17} - \gamma'\frac{ \Omega_\phi^{n} \wedge \bOmega^n}{ \Omega_\phi^n\wedge \bOmega^n} = -\gamma',\]

\[ \label{} \tag{3.18} \gamma' \frac{\Omega \wedge \Omega_\phi^{n-1} \wedge \bOmega^n}{  \Omega_\phi^n\wedge \bOmega^n} = \frac 1 n \gamma' \frac{1}{\Omega_{2i 2i+1}^{\phi}}.\]

We have thus obtained from (\ref{chernlaplaciangrad})

\[ \label{chlgloc} \tag{3.19} \begin{gathered} 0 \geq \frac{\phi_{\bar{i}} (e^f)_i}{\beta e^f} +  \frac{ \phi_{i} (e^f)_{\bar{i}}}{\beta e^f} + \frac{|\phi_{2i j}|^2 + |\phi_{2i+1 j}|^2}{ \beta \Omega^\phi_{2i 2i+1}} + \frac{|\phi_{2i \bar{j}}|^2 + |\phi_{2i+1 \bar{j}}|^2}{\beta \Omega^\phi_{2i 2i+1}}\\
 - \big( {\gamma'}^2 + \gamma'' \big) \frac{|\phi_{2i}|^2+ |\phi_{2i+1}|^2}{\Omega^\phi_{2i 2i+1}} - n \gamma' + \gamma' \frac{1}{\Omega^\phi_{2i 2i+1}}. \end{gathered} \]

Note that we may assume $\beta > 1$, otherwise we are finished. Under this assumption the first two terms in (\ref{chlgloc}) are bounded from below by a quantity not depending on $\phi$. The next two terms are positive and after fixing $\gamma$ the penultimate one is bounded from below as well. All of this allows us to rewrite the inequality (\ref{chlgloc}) as

\[ \label{somm} \tag{3.20} C(\gamma) \geq - \big( {\gamma'}^2 + \gamma'' \big) \frac{|\phi_{2i}|^2+ |\phi_{2i+1}|^2}{\Omega^\phi_{2i 2i+1}} + \gamma' \frac{1}{\Omega^\phi_{2i 2i+1}}. \] 

Now we can take, as in for example \cite{B09}, 
\[\label{gch} \tag{3.21} \gamma(t) = \frac{\log(2t+1)}{2}. \] 
Under this choice and the $C^0$ bound we have 

\[ \tag{3.22} \label{c2} C \geq C_1 \frac{|\phi_{2i}|^2+ |\phi_{2i+1}|^2}{\Omega^\phi_{2i 2i+1}} + C_2 \frac{1}{\Omega^\phi_{2i 2i+1}}.\]
From (\ref{c2}) we obtain for any fixed $j$ 
\[\Omega^\phi_{2j2j+1} \geq \frac{C_2}{C}\] which coupled with the equation (\ref{qma}) written in canonical coordinates as
\[ \label{} \tag{3.23}  \prod\limits_{i} \Omega_{2i2i+1}^\phi=e^f\]
gives 
\[ \label{newnew} \tag{3.24} \frac{1}{\Omega^\phi_{2j2j+1}} \geq \frac{C^{n-1}}{\sup_M e^f}.\]
Having the bound on $\Omega^\phi_{2i2i+1}$'s from (\ref{newnew}) we obtain the bound for $\beta$, from (\ref{c2}), at a maximum point of $\alpha$. This results in a uniform bound for $\beta$.
\end{proof}
\begin{remark} \label{graddisc}
The argument above corresponds to a gradient bound for the complex Monge-Amp\`ere equation with a background metric of non negative holomorphic bisectional curvature. Let us briefly discuss the problems in the argument above for general HKT metrics. Due to non vanishing of $d\Omega$ extra terms controlled by \[ -C\frac{1}{\Omega^\phi_{2i 2i+1}} \] appear. Hence $\gamma$ would have to
satisfy both \[ (\gamma')^2+\gamma''\geq 0 \text{ and } \gamma'>C \] which is possible only if the oscillation of $\phi$ is small compared to $C$. Furthermore the idea from \cite{B09} to exploit the terms containing squares of the pure second order derivatives does not seem to work.
This is partially explained by the fact that, even in the flat case, though the gradient can be written as \[\beta = \phi_{q_i} \phi_{\bar{q_i}}, \] its $q_j$'th derivative is not given by the Leibniz rule as this fails for the Cauchy-Riemann-Fueter operators (\ref{crf}) and (\ref{crfb}). What one obtains instead of $\phi_{q_i q_j} \phi_{\bar{q_i}}+\phi_{q_i } \phi_{\bar{q_i}q_j}$ are $\phi_{q_i } \phi_{\bar{q_i}q_j}$ and the conjugations of $\phi_{q_i q_j} \phi_{\bar{q_i}}$. Lack of orthogonality between these conjugates results in insufficient positivity to beat \[ -{\gamma'}^2 \frac{|\phi_{2i}|^2+ |\phi_{2i+1}|^2}{\Omega^\phi_{2i 2i+1}}\] - the main negative term.
\end{remark}

\section{Bound on $\partial \jpar \phi $}
In this section we bound partially the Hessian of $\phi$. More specifically we prove the following a priori estimate:

\begin{theorem}\label{le}
Let $(M,I,J,K,g)$ be a compact, connected hyperK\"ahler manifold. There exists a constant $C$ depending on $f$, $\sup_M |\phi|$ and the hyperhermitian structure $(I,J,K,g)$ such that for any solution $\phi$ of the equation (\ref{qma}) the estimate \[ \label{qhe} \tag{4.1} |\pa\jpar \phi|_g \leq C \] holds.
\end{theorem}
\begin{proof}
Let us define this time \[ \tag{4.2} \label{betlap} \eta \Omega^n = \Omega_\phi \wedge \Omega^{n-1} \] and consider the quantity \[ \alpha = \log \eta - \gamma \circ \phi, \] where the 
function $\gamma$ is as in the previous section, cf. (\ref{gch}). We note that in order to obtain (\ref{qhe}) it is sufficient to bound $\eta$ from above, at a maximal point of $\alpha$, as it is a positive quantity due to \[ \label{positive} \tag{4.3} \Omega_\phi > 0. \] This truly implies (\ref{qhe}) as in the canonical coordinates \[ \eta = \frac 1 n \big( \phi_{i \bar{i}} + n \big) \] is the constant plus the sum of the coefficients of $\pa \jpar \phi$ which in light of (\ref{positive}) are bounded from below.
  
We note that at a maximum point of $\alpha$
\[ \label{alhess} \tag{4.4} \begin{gathered} \partial \jpar \alpha = \frac{\partial \jpar \eta}{\eta} - \frac{ \pa 
\eta \wedge \jpar \eta}{\eta^2} - \gamma'' \partial \phi \wedge \jpar \phi - \gamma' \pa \jpar \phi \\
= \frac{\partial \jpar \eta}{\eta} - \Big( \big( \gamma' \big)^2 + \gamma'' \Big) \pa \phi \wedge \jpar \phi - \gamma' \Omega_\phi + \gamma' \Omega  \end{gathered} \] Here we have used the fact that at the extremal point 
\[\frac{\partial \eta}{\eta} = \gamma' \partial \phi, \]
\[\frac{\jpar \eta}{\eta} = \gamma' \jpar \phi .\]

Let us focus for the moment on the term $\partial \jpar \eta$ appearing in (\ref{alhess}). Differentiating twice the conjugation of the relation (\ref{betlap}) (recall that we work under the assumption $d \Omega = 0$) we obtain 

\[ \label{bhess} \tag{4.5} \pa \jpar \eta \wedge \Omega_\phi^{n-1} \wedge \bOmega^n = \pa \jpar \bpar \bjpar \phi \wedge \Omega_\phi^{n-1} \wedge \bOmega^{n-1}.\]
An easy calculation in the canonical coordinates shows that 
\[ \label{phifourth} \tag{4.6} \begin{gathered} \pa \jpar \bpar \bjpar \phi = \pa J^{-1} \bpar J \bpar J^{-1} \pa J \phi = \pa J^{-1} \bpar J \bpar J^{-1} \phi_i dz_i = \pa J^{-1} \bpar \phi_{i \bar
{j}} J d \overline{z_j} \wedge dz_i \\
 = \pa \phi_{i \overline{j} \overline{k}}  J^{-1} d \overline{z_k} \wedge d \overline{z_j} \wedge J^{-1} dz_i \\
 =  \phi_{i \overline{j} \overline{k} l}  dz_l \wedge J^{-1} d \overline{z_k} \wedge d \overline{z_j} \wedge J^{-1} dz_i \\
 = \Big(\phi_{i \overline{j} \overline{k} l} d \overline{z_j} \wedge J^{-1} dz_i \Big) \wedge  dz_l \wedge J^{-1} d \overline{z_k}. \end{gathered} \] 

Here we have used the relation (\ref{ahv}) and its conjugation \[ \begin{gathered} \bpar J^{-1} dz_i = \bpar J^{-1} \pa J z_i = \bpar \bjpar z_i = - \bjpar \bpar z_i = 0,\\
 \pa J^{-1} \bpar J d \overline{z_j} = \pa \jpar \bpar \overline{z_j} = \jpar \bpar \pa \overline{z_j} = 0 \end{gathered} \] and the fact that the first derivatives of the components of $J$ vanish at the point, cf. (\ref{geodesic}).

Formulas (\ref{bhess}) and (\ref{phifourth}) allows us to conclude that
\[ \label{betahessexpl} \tag{4.7} \frac{\pa \jpar \eta \wedge \Omega_\phi^{n-1}\wedge \bOmega^n}{\eta \Omega_\phi^n \wedge \bOmega^n}=\frac{\pa \jpar \bpar \bjpar \phi \wedge \Omega_\phi^{n-1} \wedge \bOmega^{n-1}}{ \eta \Omega_\phi^n \wedge \bOmega^n} = \frac{1}{\eta n^2}\sum_{l=0}^{n-1} \frac{\sum_{i=0}^{2n-1} \big( \phi_{i \overline{i} \overline{2l} 2l} + \phi_{i \overline{i} \overline{2l+1} 2l+1} \big)}{\Omega^{\phi}_{2l 2l+1}}. \]

Now we find another expression for the last quantity in (\ref{betahessexpl}). Recall that the equation (\ref{qma}) can be written in the form (\ref{cpe})
\[Pf(\Omega_{ij}^\phi) = e^f Pf(\Omega_{ij}). \] 
After taking the logarithm this reads
\[\label{logcpe} \tag{4.8} \log Pf(\Omega_{ij}^\phi) = f + \log Pf(\Omega_{ij}).\] Differentiating (\ref{logcpe}) once provides, due to (\ref{pfder}),
\[ \label{logcpeder} \tag{4.9} \frac 1 2 tr \Big( (\Omega^{ij}_\phi )(\Omega_{ij, \bar{p}}^\phi ) \Big) = f_{\bar{p}}. \] because, due to the hyperK\"ahler assumption, 
\[ \label{} \tag{4.10} 0= \bpar \Omega^n = n! Pf(\Omega_{ij})_{\bar{p}} d\overline{z_p} \wedge dz_0 \wedge ... \wedge dz_{2n-1}. \] In particular the first barred derivatives vanish locally and  not only at the fixed point. Differentiating (\ref{logcpeder}) once more, due to the formula (\ref{pfderder}), yields 
\[ \label{logcpederder} \tag{4.11} \frac 1 2 tr \Big( (\Omega^{ij}_\phi )(\Omega_{ij, \bar{p}p}^\phi ) \Big) - \frac 1 2 tr \Big( (\Omega^{ij}_\phi )(\Omega_{ij, \bar{p}}^\phi ) (\Omega^{ij}_\phi )(\Omega_{ij,p}^\phi )\Big) = f_{\bar{p}p}.\]

Summing, over $p$, the formulas (\ref{logcpederder}) give us (recall that $\big[\Omega_{ij}^\phi\big]_{i,j}$ is block diagonal)
\[ \label{fourthfinal} \tag{4.12} \begin{gathered} \frac{\sum_p \Omega^\phi_{2i2i+1,p \bar{p}}}{\Omega_{2i2i+1}^\phi} 
= \frac{1}{2} \Delta^{Ch}_{I,g} f + \frac 1 2 \Omega_\phi^{ka}\Omega^\phi_{al,p}\Omega_\phi^{lb}\Omega^\phi_{bk,\bar{p}}
 = \frac{1}{2} \Delta^{Ch}_{I,g} f + \\
\frac{1}{2} \frac{ \Omega_{2k+1 2l, p}^\phi \Omega_{2l+1 2k, \bar{p}}^\phi 
+ \Omega_{2k2l+1, p}^\phi \Omega_{2l2k+1, \bar{p}}^\phi 
- \Omega_{2k+12l+1, p}^\phi \Omega_{2l2k, \bar{p}}^\phi
- \Omega_{2k2l, p}^\phi\Omega_{2l+1 2k+1, \bar{p}}^\phi}{\Omega_{2k2k+1}^\phi \Omega_{2l2l+1}^\phi}. \end{gathered} \]

Using the hyperK\"ahler assumption and (\ref{ahv}) as well as (\ref{geodesic}) we obtain the formula
\[ \label{ofder} \tag{4.13} \begin{gathered} \sum\limits_{k <l} \Omega^\phi_{kl,\bar{p}} d \overline{z_p} \wedge dz_k \wedge dz_l = \bpar  \sum\limits_{k <l}  \Omega^\phi_{kl} dz_k \wedge dz_l = \bpar \Omega_\phi = \bpar \pa \jpar \phi \\
 = \bpar \sum\limits_{k,l} \phi_{k \bar{l}} dz_k \wedge J^{-1} d \overline{z_l} =  \sum\limits_{k,l} \phi_{k \bar{l} \bar{p}} d \overline{z_p} \wedge dz_k \wedge J^{-1} d \overline{z_l} \end{gathered} \]

and similarly 

\[ \label{ofderb} \tag{4.14} \begin{gathered} \sum\limits_{k < l} \Omega^\phi_{kl, p} J^{-1} d z_p \wedge dz_k \wedge dz_l = \bjpar  \sum\limits_{k <l}  \Omega^\phi_{kl} dz_k \wedge dz_l = \bjpar \Omega_\phi = \bjpar \pa \jpar \phi \\
 = \bjpar \sum\limits_{k,l} \phi_{k \bar{l}} dz_k \wedge J^{-1} d \overline{z_l} =  \sum\limits_{k,l} \phi_{k \bar{l} p} J^{-1} d z_p \wedge dz_k \wedge J^{-1} d \overline{z_l}. \end{gathered} \]

From (\ref{ofder}) and (\ref{ofderb}) we obtain that, for any $k$, $l$ and $p$,

\[ \label{formulas} \tag{4.15} \begin{gathered} \Omega^\phi_{2k+1 2l, p} = - \phi_{2k+1 \overline{2l+1} p} - \phi_{2l \overline{2k} p}, \\
 \Omega^\phi_{2l+1 2k, \bar{p}} = - \phi_{2l+1 \overline{2k+1} \bar{p}} - \phi_{2k \overline{2l} \bar{p}}, \\
 \Omega^\phi_{2k+1 2l+1, p} = \phi_{2k+1 \overline{2l} p} - \phi_{2l+1 \overline{2k} p}, \\
 \Omega^\phi_{2l+1 2k+1, \bar{p}} = \phi_{2k+1 \overline{2l} \bar{p}} - \phi_{2l+1 \overline{2k} \bar{p}}, \\
 \Omega^\phi_{2k 2l, p} = \phi_{2k \overline{2l+1} p} - \phi_{2l \overline{2k+1} p}, \\
 \Omega^\phi_{2l 2k, \bar{p}} = \phi_{2k \overline{2l+1} \bar{p}} - \phi_{2l \overline{2k+1} \bar{p}}. \end{gathered}\]

This gives the expression in terms of derivatives of $\phi$ for the  fourth order component obtained in (\ref{fourthfinal})  
\[ \label{fourthpositive} \tag{4.16} \begin{gathered} \frac{1}{2} \frac{ \Omega_{2k+1 2l, p}^\phi \Omega_{2l+1 2k, \bar{p}}^\phi 
+ \Omega_{2k2l+1, p}^\phi \Omega_{2l2k+1, \bar{p}}^\phi 
- \Omega_{2k+12l+1, p}^\phi \Omega_{2l2k, \bar{p}}^\phi
- \Omega_{2k2l, p}^\phi\Omega_{2l+1 2k+1, \bar{p}}^\phi}{\Omega_{2k2k+1}^\phi \Omega_{2l2l+1}^\phi} \\
 = \frac{1}{2} \frac{ |\phi_{2k+1 \overline{2l+1} p} + \phi_{2l \overline{2k} p}|^2 
+ |\phi_{2l+1 \overline{2k+1} p} + \phi_{2k \overline{2l} p}|^2}{\Omega_{2k2k+1}^\phi \Omega_{2l2l+1}^\phi} \\
+\frac 1 2 \frac{ |\phi_{2k+1 \overline{2l} p} - \phi_{2l+1 \overline{2k} p}|^2
+ |\phi_{2k \overline{2l+1} p} - \phi_{2l \overline{2k+1} p}|^2}{\Omega_{2k2k+1}^\phi \Omega_{2l2l+1}^\phi}. \end{gathered}\]

From (\ref{fourthpositive}) we see that the quantity from (\ref{fourthfinal}) satisfies

\[ \label{prelimbound} \tag{4.17} \frac{\sum_p \Omega^\phi_{2i2i+1,p \bar{p}}}{\Omega_{2i2i+1}^\phi} 
\geq \frac{1}{2} \Delta^{Ch}_{I,g} f \geq - C(f). \]

Finally observe that from (\ref{phifourth}), much like in (\ref{ofder}) and (\ref{ofderb}), it is easy to see that 
\[\label{fourthexchange} \tag{4.18}  \phi_{i \overline{i} \overline{2l} 2l} + \phi_{i \overline{i} \overline{2l+1} 2l+1} = \Omega^\phi_{2l2l+1,i \bar{i}}\] for any $i,l$. 

Having this we return to the estimation of $\Omega^\phi_{ij}$'s. At a maximum point of $\alpha$ we have 

\[ \label{maxalphlap} \tag{4.19} \begin{gathered} 0 \geq \frac{1}{2n} \Delta^{Ch}_{I,g_\phi} \alpha = \frac{\pa \jpar \alpha \wedge \Omega_\phi^{n-1} \wedge \bOmega^n}{\Omega_\phi^n\wedge \bOmega^n} \\
= \frac{\pa \jpar \eta \wedge \Omega_\phi^{n-1} \wedge \bOmega^n}{ \eta \Omega_\phi^n\wedge \bOmega^n}  - \big( (\gamma')^2 + \gamma'' \big) \frac{ \pa \phi \wedge \jpar \phi \wedge \Omega_\phi^{n-1} \wedge \bOmega^n}{ \Omega_\phi^n\wedge \bOmega^n} \\
- \gamma'\frac{ \Omega_\phi^{n} \wedge \bOmega^n}{  \Omega_\phi^n\wedge \bOmega^n} + \gamma' \frac{\Omega \wedge \Omega_\phi^{n-1} \wedge \bOmega^n}{  \Omega_\phi^n\wedge \bOmega^n}. \end{gathered} \]

We may assume $\eta > 1$, otherwise we are done. By (\ref{betahessexpl}), (\ref{fourthexchange}) and (\ref{prelimbound}) the first term on the right hand side of (\ref{maxalphlap}) is bounded from below. The same holds for the third term of (\ref{maxalphlap}) as $\gamma$ is chosen as in (\ref{gch}). 
This choice of $\gamma$ ensures also the positivity of the coefficients of the two remaining terms on the right hand side of (\ref{maxalphlap}). This means we can rewrite the inequality (\ref{maxalphlap}) as

\[ \tag{4.20} \label{c1} C \geq C_1 \frac{|\phi_{2i}|^2+ |\phi_{2i+1}|^2}{\Omega^\phi_{2i 2i+1}} + C_2 \frac{1}{\Omega^\phi_{2i 2i+1}}\]

for positive constants $C_1$, $C_2$. This allows us, as in the previous section, to obtain the bounds on $\Omega^\phi_{2i 2i+1}$'s at a maximum point of $\alpha$. This in turn gives us a bound on $\eta$, which is a multiple of the sum of $\Omega^\phi_{2i 2i+1}$'s, at a maximum point of $\alpha$ which in turn yields the uniform bound on $\eta$ itself.
\end{proof}
\begin{remark}\label{lapdisc}
Let us note that for the general HKT metric  the presence of terms coming from differentiating $\Omega$ in (\ref{alhess}) significantly complicates the computations.
 Instead, one has to bound the term \[\frac{ \pa \eta \wedge \jpar \eta}{\eta^2}\] in (\ref{alhess}). Unfortunately as one can easily see the quantity 
 \[ \frac{ \pa \eta \wedge \jpar \eta \wedge \Omega_\phi^{n-1} \wedge \bOmega^n}{\eta^2 \Omega^n \wedge \bOmega^n} = \frac{1}{\eta^2 n^3} \sum_i \frac{|\phi_{j \bar{j} i}|^2}{|\Omega_{ii+(-1)^i}^\phi|}\] is only bounded by twice the quantity (\ref{betahessexpl}) as can bee seen from what we have obtained in (\ref{fourthpositive}).  
\end{remark}
\section{Full $C^2$ estimate}

This section fully exploits the fact that under the assumptions of Theorem \ref{mainthm} \[ \label{} \tag{5.1} \nabla:= \ob = \nabla^{LC} = \nabla^{Ch}_{I,g}. \]
As we shall see this coupled with the previous a priori bounds suffices to bound the full Hessian of $\phi$. The $C^2$ a priori estimate reads as follows:
\begin{theorem}\label{c2b}
Let $(M,I,J,K,g)$ be a compact, connected hyperK\"ahler manifold. There exists a constant $C$ depending on $f$, $\sup_M |\phi|$ and the hyperhermitian structure $(I,J,K,g)$ such that for any solution $\phi$ of the equation (\ref{qma}) the estimate \[ \label{fhe} \tag{5.2} | \nabla^2 \phi|_g \leq C \] holds.
\end{theorem}
\begin{proof}
We wish to estimate the quantity $\theta$ being, this time, defined for any $x \in M$ as
\[ \label{lamax} \tag{5.3} \theta(x) = \lambda_{max}(x) := \sup\limits_{X \in T_x M, \: |X|_g = 1} g(\nabla_X \nabla \phi,X) =  \sup\limits_{X \in T_x M, \: |X|_g = 1} \big(\nabla^2 \phi \big) (X,X)\]
as was done originally in \cite{Bnotes} in the case of the complex Monge-Amp\`ere equation on K\"ahler manifolds. This is sufficient for the bound on the full Hessian (\ref{fhe}) because $\theta$ is the maximum eigenvalue of the Hessian at the point $x$. More precisely, the sum of all the eigenvalues, being the Laplacian, is bounded from below by 
\[ \label{lbl} \tag{5.4} \frac{1}{2n} \Delta^{Ch}_{I,g} \phi \geq - \frac{\Omega \wedge \Omega^{n-1}}{\Omega^n} = -1\] 
as $\Omega_\phi > 0 $. Once the sum is under control from below and $\lambda_{max}$ is bounded from above we obtain the lower bound for the smallest eigenvalue, $\lambda_{min}$, and consequently we get both sided bounds for all the entries of the matrix of $\nabla^2 \phi$.

Consider the quantity $\alpha$ this time given by 
\[ \label{prelimal} \tag{5.5} \alpha = \theta + \frac 1 4 |d\phi|_g^2. \]

Since we obtained the gradient bound in Theorem \ref{c1b} it is enough to estimate $\alpha$ at a maximum point $p \in M$.

In this section it will be customary to introduce also the real coordinates \[\label{realcoord} \tag{5.6} z_i = t_i + \ii t_{2n+i} \] for $i=0,...,2n-1$, different from the one introduced in (\ref{rc}) as can be seen from Remark \ref{ccvrc}.

As the quantity (\ref{prelimal}) is in general non smooth, due to (\ref{lamax}) not being smooth, we extend a fixed vector \[ \label{} \tag{5.7} X = X^j \partial_j(p) \in T_p M,\] realizing the supremum in the definition of $\theta(p)$, to a constant coefficient local
vector field \[ \label{} \tag{5.8} X = X^j \partial_j.\] This $X$ is fixed for the rest of the proof. Consider instead of (\ref{prelimal}) the quantity \[ \label{al} \tag{5.9} \tilde{\alpha} = \frac{\tilde{\theta}}{|X|_g^2} + \frac 1 4 |d\phi|_g^2,\] where 
\[ \tag{5.10} \label{lamaxmod} \tilde{\theta} = \big( \nabla^2 \phi \big) (X,X). \] 

Observe that \[\label{lmvlm} \tag{5.11} \frac{\tilde{\theta}}{|X|_g^2} \leq \theta \] and \[ \label{lmvlm'} \tag{5.12} \tilde{\theta}(p) = \theta(p)\] which means that also the quantity (\ref{al}) attains a maximum at $p$. We may assume that
\[ \label{} \tag{5.13} \theta(p) \geq 0\] since otherwise we are done.

We have the following expression in the introduced coordinates (\ref{realcoord})
\[ \label{hessexpre} \tag{5.14} \nabla^2 \phi = \nabla \phi_{t_j} dt_j = \phi_{t_i t_j} dt_i \otimes dt_j - \Gamma_{ik}^j \phi_{t_j} dt_i \otimes dt_k,\] where $\Gamma_{ji}^k$ are Christoffel symbols in the real frame $\partial_{t_i}$. From (\ref{hessexpre}) we find that 
\[ \label{lamaxmodexpre} \tag{5.15} \tilde{\theta} = D^2_X \phi - \Gamma_{ik}^j \phi_{t_j} X^i X^k, \] 
where $D$ denotes the flat connection in coordinates (\ref{realcoord}). Our goal is to exploit the estimate 
\[ \label{lapineq} \tag{5.16} \begin{gathered} 0 \geq \frac{1}{2n} \Delta^{Ch}_{I,g_\phi} \tilde{\alpha} \\ = 
\frac{\pa \jpar \tilde{\theta} \wedge \Omega_\phi^{n-1} \wedge \bOmega^n}{\Omega_\phi^n \wedge \bOmega^n} 
- \tilde{\theta}\frac{\pa \jpar \left( {|X|_g^2} \right) \wedge \Omega_\phi^{n-1} \wedge \bOmega^n}{\Omega_\phi^n \wedge \bOmega^n}
+ \frac{\pa \jpar \frac 1 4 | d \phi |_g^2 \wedge \Omega_\phi^{n-1} \wedge \bOmega^n}{\Omega_\phi^n \wedge \bOmega^n},\end{gathered} \]
where we have used the fact that at the point $p$ \[\pa |X|_g^2 = \jpar |X|_g^2 = 0. \]

As we already noticed in the previous section, in the canonical coordinates, we have at $p$
\[ \label{lamaxmodder} \tag{5.17} \frac{\pa \jpar \tilde{\theta} \wedge \Omega_\phi^{n-1} \wedge \bOmega^n}{\Omega_\phi^n \wedge \bOmega^n} 
= \frac 1 n \frac{\tilde{\theta}_{2p\overline{2p}} + \tilde{\theta}_{2p+1\overline{2p+1}}}{\Omega_{2p2p+1}^\phi}.\]

Differentiating the expression (\ref{lamaxmodexpre}) for $\tilde{\theta}$ we obtain
\[ \label{lamaxmodderest} \tag{5.18}  \tilde{\theta}_{p \bar{p}} = D^2_X \phi_{p \bar{p}} - \Gamma_{ik,p \bar{p}}^j \phi_{t_j} X^i X^k  - \Gamma_{ik,p}^J \phi_{t_j \bar{p}} X^i X^k - \Gamma_{ik, \bar{p}}^j \phi_{t_j p} X^i X^k \geq D^2_X \phi_{p \bar{p}} - C(\tilde{\theta} +1 ).\]
In the estimation (\ref{lamaxmodderest}) we have used the facts that $\Gamma_{ij}^k$ vanish at the point $p$ (recall (\ref{chrj}) and (\ref{geodesic})), $\Gamma_{ij}^k$'s derivatives depend only on the derivatives of the initial metric $g$, the gradient of $\phi$ is bounded and \[ \label{mixeddercontrol} \tag{5.19} |\phi_{t_it_j}|<C(1 + \tilde{\theta}). \]

Since we know from Section 4, (\ref{qhe}), that \[ \label{qheh} \tag{5.20} \frac 1 C \leq \Omega^\phi_{2i2i+1} \leq C \] we can estimate the quantity (\ref{lamaxmodder}) by applying (\ref{lamaxmodderest}) and (\ref{qheh})
\[ \label{lamaxmodderestfin} \tag{5.21} \frac 1 n \frac{\tilde{\theta}_{2p\overline{2p}} + \tilde{\theta}_{2p+1\overline{2p+1}}}{\Omega_{2p2p+1}^\phi} 
\geq \frac 1 n  \frac{D^2_X \phi_{2p \overline{2p}} + D^2_X \phi_{2p+1 \overline{2p+1}}}{\Omega^\phi_{2p2p+1}} - C(\tilde{\theta} + 1). \]

In order to deal with the last remaining terms involving derivatives of $\phi$ in (\ref{lamaxmodderestfin}) let us differentiate the equation (\ref{logcpe}) twice in the direction of $X$ obtaining
\[ \label{cpederder} \tag{5.22} \frac 1 2 tr \Big( (\Omega^{ij}_\phi )(D^2_X \Omega_{ij}^\phi ) \Big) - \frac 1 2 tr \Big( (\Omega^{ij}_\phi )(D_X \Omega_{ij}^\phi ) (\Omega^{ij}_\phi )(D_X \Omega_{ij}^\phi )\Big) = D^2_X f + D^2_X \log Pf(\Omega_{ij}). \]

Rewriting the quantity in (\ref{cpederder}) explicitly gives 
\[\label{cpederderexpl} \tag{5.23} \begin{gathered} \frac{ D^2_X \Omega^\phi_{2i2i+1}}{\Omega_{2i2i+1}^\phi} 
= D^2_X f + D^2_X \log Pf(\Omega_{ij}) + \frac 1 2 \Omega_\phi^{ka}D_X \Omega^\phi_{al}\Omega_\phi^{lb} D_X \Omega^\phi_{bk} \\
= D^2_X f + D^2_X \log Pf(\Omega_{ij}) \\
+ \frac{1}{2} \frac{ D_X \Omega_{2k+1 2l}^\phi D_X \Omega_{2l+1 2k}^\phi 
+ D_X \Omega_{2k2l+1}^\phi D_X \Omega_{2l2k+1}^\phi}{\Omega_{2k2k+1}^\phi \Omega_{2l2l+1}^\phi} \\ 
- \frac 1 2 \frac{D_X \Omega_{2k+12l+1}^\phi D_X \Omega_{2l2k}^\phi
+ D_X \Omega_{2k2l}^\phi D_X \Omega_{2l+1 2k+1}^\phi}{\Omega_{2k2k+1}^\phi \Omega_{2l2l+1}^\phi}. \end{gathered}\]

From the formulas we obtained in (\ref{formulas}) we have the expression for (\ref{cpederderexpl}) in terms of the derivatives of $\phi$ as follows
\[ \label{cpederderexplfin} \tag{5.24} \begin{gathered} \frac{1}{2} \frac{ D_X \Omega_{2k+1 2l}^\phi D_X \Omega_{2l+1 2k}^\phi 
+ D_X \Omega_{2k2l+1}^\phi D_X \Omega_{2l2k+1}^\phi}{\Omega_{2k2k+1}^\phi \Omega_{2l2l+1}^\phi} \\ 
- \frac 1 2 \frac{D_X \Omega_{2k+12l+1}^\phi D_X \Omega_{2l2k}^\phi
+ D_X \Omega_{2k2l}^\phi D_X \Omega_{2l+1 2k+1}^\phi}{\Omega_{2k2k+1}^\phi \Omega_{2l2l+1}^\phi} \\
= \frac{1}{2} \frac{ |D_X \phi_{2k+1 \overline{2l+1}} + D_X\phi_{2l \overline{2k}}|^2 
+ |D_X \phi_{2l+1 \overline{2k+1}} + D_X \phi_{2k \overline{2l} }|^2}{\Omega_{2k2k+1}^\phi \Omega_{2l2l+1}^\phi} \\
+ \frac 1 2 \frac{|D_X \phi_{2k+1 \overline{2l} } - D_X \phi_{2l+1 \overline{2k} }|^2
+ |D_X \phi_{2k \overline{2l+1} } - D_X \phi_{2l \overline{2k+1} }|^2}{\Omega_{2k2k+1}^\phi \Omega_{2l2l+1}^\phi} \end{gathered}\]
which is seen to be non negative. We also note that \[ \label{} \tag{5.25} \Omega_{ij}^\phi = \Omega_{ij} + (- \phi_{i \bar{k}}J^{\bar{k}}_j + \phi_{j \bar{k}}J^{\bar{k}}_i)\] which gives (recall (\ref{geodesic})),
\[ \label{omphiderder} \tag{5.26} D^2_X \Omega^\phi_{2i2i+1} = D_X^2 \Omega_{2i2i+1} + D^2_X \phi_{2i \overline{2i}} + D^2_X \phi_{2i+1 \overline{2i+1}} + (- \phi_{2i \bar{k}} D^2_X J^{\bar{k}}_{2i+1} + \phi_{2i+1 \bar{k}} D^2_X J^{\bar{k}}_{2i}).\] 

Applying (\ref{omphiderder}) in (\ref{cpederderexpl}) coupled with (\ref{mixeddercontrol}) and (\ref{cpederderexplfin}) provides 
\[ \label{almostthere} \tag{5.27} \frac{D^2_X \phi_{2p \overline{2p}} + D^2_X \phi_{2p+1 \overline{2p+1}}}{\Omega^\phi_{2p2p+1}} \geq- C(\tilde{\theta} + 1).\]
Finally, applying (\ref{almostthere}) in (\ref{lamaxmodderestfin}) gives
\[ \label{firsttermbound} \tag{5.28} \frac 1 n \frac{\tilde{\theta}_{2p\overline{2p}} + \tilde{\theta}_{2p+1\overline{2p+1}}}{\Omega_{2p2p+1}^\phi} 
\geq - C(\tilde{\theta} + 1) \] providing the lower bound for the first quantity in (\ref{lapineq}).

As for the third factor of (\ref{lapineq}) we recall from the computations for the gradient estimate, cf. (\ref{gradderdersec}), that 

\[ \label{gradderdersecsec} \tag{5.29} \begin{gathered} \frac{\pa \jpar \frac 1 4 | d \phi |_g^2 \wedge \Omega_\phi^{n-1} \wedge \bOmega^n}{\Omega_\phi^n \wedge \bOmega^n} \\
 = \frac{1}{n} \Big( \frac{\phi_{\bar{i}} e^f_i}{e^f} +  \frac{ \phi_{i} e^f_{\bar{i}}}{e^f} + \frac{\phi_{\bar{j}\bar{k}} \phi_{jk}}{|\Omega^\phi_{k k+(-1)^k}|} + \frac{\phi_{i\bar{j}} \phi_{ \bar{i}j}}{|\Omega^\phi_{i i+(-1)^i}|} \Big). \end{gathered} \]

By (\ref{qheh}) and (\ref{c1e}) we obtain from (\ref{gradderdersecsec}) the bound
\[ \label{finishing} \tag{5.30} \frac{\pa \jpar \frac 1 4 | d \phi |_g^2 \wedge \Omega_\phi^{n-1} \wedge \bOmega^n}{\Omega_\phi^n \wedge \bOmega^n} \geq -C' + C(|\phi_{ij}|^2+ |\phi_{i\bar{j}}|^2).\]
We note that 
\[ \label{cplexreal} \tag{5.31} |\phi_{ij}|^2+ |\phi_{i\bar{j}}|^2 \geq C \tilde{\theta}^2.\]
Applying  (\ref{cplexreal}) in (\ref{finishing}) gives us
\[\label{secondcomponent} \tag{5.32} \frac{\pa \jpar \frac 1 4 | d \phi |_g^2 \wedge \Omega_\phi^{n-1} \wedge \bOmega^n}{\Omega_\phi^n \wedge \bOmega^n} \geq -C' + C \tilde{\theta}^2.\]

The second term in (\ref{lapineq}) can be easily seen to satisfy
\[\label{vectderv} \tag{5.33} -\tilde{\theta}\frac{\pa \jpar \left( {|X|_g^2} \right) \wedge \Omega_\phi^{n-1} \wedge \bOmega^n}{\Omega_\phi^n \wedge \bOmega^n} \geq -C \tilde{\theta}.\] This is because after rewriting this term in coordinates as in (\ref{lamaxmodder}) and applying (\ref{qheh}) we observe it depends only on $\tilde{\theta}$ and second derivatives of the metric $g$ alone (as the coefficients of $X$ are constant and we are computing in normal coordinates).

The estimations (\ref{firsttermbound}), (\ref{secondcomponent}) and (\ref{vectderv}) applied in (\ref{lapineq}) deliver
\[ \label{grandfin} \tag{5.34} 0 \geq \frac{1}{2n} \Delta^{Ch}_{I,g_\phi} \tilde{\alpha} \geq C \tilde{\theta}^2 - C' \tilde{\theta} - C''.\]
Inequality (\ref{grandfin}) provides the desired estimate on $\tilde{\theta}$ in terms of $C$, $C'$ and $C''$. The desired bound on $\theta$ follows from that.
\end{proof}
\section{Proof of Theorem \ref{mainthm}}

As we advocated in the introduction having Theorem \ref{c1b}, Theorem \ref{c2b} and using Corollary 5.7 in \cite{AV10} (or Theorem 1.1.13 of \cite{AS17} or Theorem A in \cite{Sr19}), one obtains the $C^{2,\alpha}$ a priori estimate for the solutions of (\ref{qma}) for some $\alpha\in(0,1)$. The rest of the proof are completely standard. For the convenience of the reader we sketch them and we refer to \cite{A13} Section 5 for a the detailed discussion.

Turning to the proof of Theorem \ref{mainthm}. For the given $f$ satisfying (\ref{norm}) we set up the continuity path
\[ \label{cp} \tag{6.1} (\Omega + \pa \jpar \phi)^n = \big( te^f +(1-t) \big) \Omega^n\] for $t \in [0,1]$. In order to prove Theorem \ref{mainthm} it is enough to show that the set $S$ of $t \in [0,1]$, for any fixed $k \geq 1$ and $\alpha \in (0,1)$, such that there exists $\phi \in C^{k+2,\alpha}$ solving (\ref{cp}) is both open and closed. 

Openness in our setting is completely standard. One has to prove that for a fixed $\phi$ the operator
\[ \label{fo} \tag{6.2} U^{k+2,\alpha} \ni \psi \longmapsto \frac{(\Omega+ \pa \jpar \phi + \pa \jpar \psi)^n}{\Omega^n} \in V^{k,\alpha}\] has an open image. In (\ref{fo}) the Banach manifolds, with the induced H\"older norms, are given by
\[\label{} \tag{6.3} \begin{gathered} U^{k+2,\alpha} := \{ \phi \in C^{k+2,\alpha}(M) \: | \: \int_M \phi \Omega^n \wedge \bOmega^n = 0 \} \\ 
V^{k,\alpha} := \{ F \in C^{k,\alpha}(M) \: | \: \int_M F \Omega^n \wedge \bOmega^n = \int_M \Omega^n \wedge \bOmega^n \}. \end{gathered}\] The fact that (\ref{fo}) has an open image was proven in Proposition 5.1 of \cite{A13} and relies on the theory of linear elliptic operators.

As for the closedness it is enough to know that once $t_i \in S$ are such that $t_i \rightarrow t$ then $t \in S$ as well. For any $t_i$ let us  take $\phi_{t_i}$ solving (\ref{cp}) normalized by \[ \label{normsol} \tag{6.4} \sup_M \phi_{t_i}=0 .\]
Once we know that the sequence $\phi_{t_i}$ is bounded in $C^{k+3,\alpha}$ the Kondrakov theorem yields a subsequence converging in a $C^{k+3}$ norm (and hence in $C^{k+2, \alpha}$ norm) to the solution $\phi$ of (\ref{cp}). All we need then is to have 
 a priori estimates for the solutions of (\ref{qma}) normalized by (\ref{normsol}) up to the order $C^{k+3,\alpha}$. From Theorem \ref{c1b}, Theorem \ref{c2b} and Corollary 5.7 in \cite{AV10} we have the $C^2$ estimate. Applying now the Evans--Krylov theorem, 
 cf. \cite{E82}, to the operator  \[ \label{} \tag{6.5} \log \bigg( \frac{(\Omega+ \pa \jpar \phi)^n}{\Omega^n} \bigg)\] defined on those functions for which the hyperhermitian matrix associated to $\Omega+ \pa \jpar \phi$ is positive we obtain a $C^{2,\alpha}$ estimate for some fixed $\alpha$. From this the standard procedure of bootstrapping provides the bounds of any higher order as in Section 17.5 of \cite{GT01}.
 
\begin{remark} \label{comment}
For simplicity of presentation (and calculations for obtaining a priori estimates) we stated Theorem \ref{mainthm} in the setting when the initial metric $g$ is already HK. Actually, our method works equally well just under the assumption that the hypercomplex manifold $(M,I,J,K)$ admits some compatible hyperK\"ahler metric $g'$ and the initial hyperhermitian metric $g$ is arbitrary (in particular HKT). 
In that setting no new, essential, complications arise while performing a priori estimates, provided we still define the test quantities using $g'$. This is because all the new terms are estimable from below by 
\[- \frac{C}{\delta^\sigma} \cdot \frac{1}{\Omega^\phi_{2i 2i+1}},\]
where $C$ is a constant depending on the curvature of the initial metric $g$, $\sigma \in \{\frac{1}{2},1\}$ and $\delta$ is the test quantity we are trying to estimate at the moment. Thus, we can always assume $\delta$ makes the above term arbitrarily small in comparison with 
\[ \gamma' \cdot \frac{1}{\Omega^\phi_{2i 2i+1}},\]
for $\gamma$ which is chosen in the sections above, since otherwise we already obtain a bound on the test quantity. 
\end{remark}

\noindent {FACULTY OF MATHEMATICS AND COMPUTER SCIENCE\\
OF JAGIELLONIAN UNIVERSITY \\
\L OJASIEWICZA  6 \\
30-348, KRAK\'OW \\
POLAND \\
\textit{E-mail address:} Slawomir.Dinew@im.uj.edu.pl\\

\noindent {FACULTY OF MATHEMATICS AND COMPUTER SCIENCE\\
OF JAGIELLONIAN UNIVERSITY \\
\L OJASIEWICZA  6 \\
30-348, KRAK\'OW \\
POLAND \\
\textit{E-mail address:} Marcin.Sroka@im.uj.edu.pl}

\end{document}